\documentclass[11pt]{article}
\usepackage{amsmath}
\usepackage{amsfonts}
\usepackage{graphicx}
\usepackage{setspace}
\usepackage{amsmath}
\usepackage{amssymb}
\usepackage{latexsym}
\usepackage{amsmath, amsfonts,amssymb, amsthm, euscript,makeidx,color,mathrsfs}

\usepackage[numbers,sort&compress]{natbib}

\def\sqr#1#2{{\vcenter{\vbox{\hrule height.#2pt
              \hbox{\vrule width.#2pt height#1pt \kern#1pt \vrule width.#2pt}
              \hrule height.#2pt}}}}
\newtheorem {theorem}{Theorem}[section]
\newtheorem {lemma}[theorem]{{\bf Lemma}}

\newtheorem {proposition}[theorem]{{\bf Proposition}}
\theoremstyle{remark}
\newtheorem {remark}{{\bf Remark}}[section]

\theoremstyle{definition}
\newtheorem {definition}{{\bf Definition}}[section]
\theoremstyle{plain} \numberwithin {equation}{section}

\numberwithin{assumption}{section}

\def\dbE{{\mathbb{E}}}
\def\dbF{{\mathbb{F}}}

\def\dbN{{\mathbb{N}}}

\def\dbR{{\mathbb{R}}}

\def\Om{\Omega}
\def\om{\omega}

\usepackage{hyperref}
\hypersetup{
    colorlinks=true,
    linkcolor=blue,
    urlcolor=cyan,
    citecolor=cyan
}

\allowdisplaybreaks

\begin{document}

\title{\bf Null controllability for  stochastic fourth order parabolic equations\footnote{This work is supported by the NSF of China under grants 12025105, 11971334, and by the Chang Jiang Scholars Program from the Ministry of Education of the Peoples Republic of China.}}

\author{
    Qi L\"{u}
    \footnote{School of Mathematics, Sichuan University, Chengdu, P. R. China. Email: lu@scu.edu.cn. }
    ~~~ and ~~~
    Yu Wang
    \footnote{School of Mathematics, Sichuan University, Chengdu, P. R. China.
    Email: yuwangmath@163.com.}
}

\date{}

\maketitle

\begin{abstract}
    We establish the null controllability for linear stochastic fourth order parabolic equations.
    Utilizing the duality argument, the null controllability is reduced to the observability for backward  fourth order stochastic parabolic equations, and the desired observability estimate is obtained by a new global Carleman estimate.
    Our Carleman estimate  is based on a new fundamental identity for a stochastic fourth order parabolic operator.
\end{abstract}



% put page count to the bottom
\pagestyle{plain}

\section{Introduction}

Let $T>0$ and $(\Omega, \mathcal{F},
\mathbf{F}, \mathbb{P})$ with $\mathbf
F=\{\mathcal{F}_{t}\}_{t \geq 0}$ be a complete
filtered probability space on which a
one-dimensional standard Brownian motion
$\{W(t)\}_{t \geq 0}$ is defined and
$\mathbf{F}$ is the natural filtration generated
by $W(\cdot)$, augmented by all the $\mathbb{P}$
null sets in $\mathcal{F}$.  Let $H$ be a Banach
space. Denote   by
$L^2_{\mathcal{F}_t}(\Omega;H)$ the space
of all $\mathcal{F}_t$-measurable random
variables $\xi$ such that
$\dbE|\xi|_H^2<\infty$; by
$L_{\mathbb{F}}^{2}(0, T ; H)$ the space
consisting of all $H$-valued $\mathbf F$-adapted
processes $X(\cdot)$ such that
$\mathbb{E}\bigl(|X(\cdot)|_{L^{2}(0, T ;
H)}^{2}\bigr)<\infty$; by
$L_{\mathbb{F}}^{\infty}(0, T ; H)$  the space
consisting of all $H$-valued $\mathbf F$-adapted
bounded processes; and by
$L_{\mathbb{F}}^{2}(\Omega ; C[0, T] ; H))$ the
space consisting of all $H$-valued $ \mathbf{F}
$-adapted continuous processes $X(\cdot)$ such
that $\mathbb{E}\bigl(|X(\cdot)|_{C([0, T] ;
H)}^{2}\bigr)<\infty$. All these spaces are
Banach spaces with the cononical norms (e.g.,
\cite[Section 2.6]{Lue2021a}).

Let $G \subset \mathbb{R}^{n}$ ($ n\in\dbN $) be
a bounded domain with a $C^\infty$ boundary
$\partial G$. Set $Q=(0, T) \times G$ and $
\Sigma= (0, T) \times \partial G$. Let $G_{0}
\subset G$ be a nonempty open subset.  As usual,
$ \chi_{G_{0}} $   denotes the characteristic
function of $ G_{0} $.

Consider  the following control system:
\begin{equation}\label{eq.forwardEquation}
    \left\{\begin{alignedat}{2}
        &d y+\Delta^2 y d t=(a_{1} y+\chi_{G_{0}} u+a_{3} v) d t+(a_{2} y+v) d W(t) && \quad \text { in } Q, \\
        &y=\Delta y=0 &&\quad \text { on }\Sigma, \\
        &y(0)=y_{0} &&\quad \text { in } G,
    \end{alignedat}\right.
\end{equation}
where $y_0\in L^2(G)$, $a_{1}, a_{2}, a_{3} \in
L^{\infty}_{\mathbb{F}}( 0,T;L^{\infty}(G) )$
and $(u, v) \in L_{\mathbb{F}}^{2}(0, T ;
L^{2}(G_{0})) \times$ $L_{\mathbb{F}}^{2}(0, T ;
L^{2}(G))$.  In \eqref{eq.forwardEquation}, $y$
is the state and $(u,v)$ are the controls. 
By the classical wellposedness result for
stochastic evolution equations, we know that
\eqref{eq.forwardEquation}  admits a unique weak
solution $y\in
L^2_\dbF(\Om;C([0,T];L^2(G)))\times
L^2_\dbF(0,T;H^2(G)\cap H_0^1(G))$ (e.g.,
\cite[Theorem 3.24]{Lue2021a}).

\begin{remark}
	The
	term $a_3v$ reflects the effect of the control
	$v$ in the diffusion term to the drift term. It is the side effect of the control $v$, i.e.,  once we put a control $v$ in the diffusion term, $a_3v$ appears passively.  This will lead to some technical difficulties in the study of the null controllability of \eqref{eq.forwardEquation}.
\end{remark}

When $u=v=0$, the equation
\eqref{eq.forwardEquation}  is widely used to
describe some fast diffusion phenomenon under
stochastic perturbations (e.g.,
\cite{Cook1970,Hohenberg1977,Langer1971}).
Further, it is the linearized version of
stochastic Cahn-Hilliard equation, which is used
to model phase separation phenomena in a melted
alloy that is quenched to a temperature at which
just two distinctive phases can exist steadily
(e.g., \cite{CardonWeber2001}). By acting
control on the system, we aim to modify the
diffusion of the material. It
is an important goal
to drive the state to the rest. This leads to
the following definition.
\begin{definition}
The system \eqref{eq.forwardEquation} is said to
be {\it null controllable} at time $T$ if for
any $y_{0} \in L^{2}(G)$, there exists a pair
$(u, v) \in L_{\mathbb{F}}^{2}(0, T ;
L^{2}(G_{0})) \times$ $L_{\mathbb{F}}^{2}(0, T ;
L^{2}(G))$ such that the corresponding solution
to \eqref{eq.forwardEquation} fulfills that
$y(T)=0$, $\mathbb{P}$-a.s.
\end{definition}

The main result of this paper is the following.
\begin{theorem}\label{thm.nullControllableOfForwardEquation}
    System \eqref{eq.forwardEquation} is null controllable at any time $ T>0. $
\end{theorem}

\begin{remark}
Like in \cite{Tang2009}, due to the randomness, extra control is required in the diffusion term. Hence, the control of \eqref{eq.forwardEquation}  is  a pair of function $ (u, v) $.
It would be quite interesting to study the null controllability for system \eqref{eq.forwardEquation} with only one control in the drift.
\end{remark}

To prove the null controllability of \eqref{eq.forwardEquation}, we introduce the adjoint equation:
\begin{equation}\label{eq.adjointBackwardEquation}
    \left\{\begin{alignedat}{2}
        &d z-\Delta^2 z d t=-(a_{1} z+a_{2} Z) d t+Z d W(t) && \quad \text { in } Q, \\
        &z=\Delta z=0 && \quad \text { on } \Sigma, \\
        &z(T)=z_{T} && \quad \text { in } G.
    \end{alignedat}\right.
\end{equation}
By the classical wellposedness result for backward stochastic evolution equations, we know that  \eqref{eq.adjointBackwardEquation} admits a unique weak solution $(z,Z)\in \big(L^2_\dbF(\Omega;C([0,T];L^2(G))\times L^2_\dbF(0,T;H^2(G)\cap H_0^1(G))\big)\times L^2_\dbF(0,T;L^2(G))$ (e.g., \cite[Theorem 4.11]{Lue2021a}).

By a standard duality argument, Theorem \ref{thm.nullControllableOfForwardEquation} is implied by the following observability estimate (e.g., \cite[Theorem 7.17]{Lue2021a}).
\begin{theorem}\label{thm.observabilityForBackwardEquation}
    There exists a constant $ C>0 $ such that the solution $ (z,Z) $ to system \eqref{eq.adjointBackwardEquation} satisfies
    \begin{equation*} 
        |z(0)|_{L^{2}(G)} \leq C e^{C (1+r_{1}^{2})(T^{-1}+1)}
        (| z|_{L_{\mathbb{F}}^{2}(0, T ; L^{2}(G_{0}))}+|a_{3} z+Z|_{L_{\mathbb{F}}^{2}(0, T ; L^{2}(G))})
        , \;\, \forall  z_{T} \in L_{\mathcal{F}_{T}}^{2}(\Omega ; L^{2}(G)),
    \end{equation*}
    where $ r_{1} =
    |a_{1}|_{L^{\infty}_{\mathbb{F}}( 0,T;L^{\infty}(G) )}
    + |a_{2}|_{L^{\infty}_{\mathbb{F}}( 0,T;L^{\infty}(G) )}
    + |a_{3}|_{L^{\infty}_{\mathbb{F}}( 0,T;L^{\infty}(G) )} $.
\end{theorem}

Here and in what follows, $C$ stands for a generic positive constant depending only on $ G $ and $ G_{0} $ whose value may vary from line to line.

\begin{remark}
    In \cite{Tang2009}, the authors obtain observability estimates for second order stochastic parabo\-lic equation.
    The constant  in their observability inequality is  $ Ce^{CT^{-4}} $ when $ T $ is small. In our case, it is $ Ce^{CT^{-1}} $.
\end{remark}

Controllability problems for deterministic
parabolic  equations are extensively studied.
Particularly, we refer readers to
\cite{Carreno2016,Cerpa2011,Cerpa2017,
Cerpa2015,Gao2016,Guerrero2019,Guzman2020,Zhou2012} for
controllability problems of fourth order
parabolic equations. On the other hand, although
controllability problems for stochastic PDEs are
not extensively studied, some progresses are
obtained in recent years (see
\cite{Gao2015,Liu2014, Lue2021b,
Lue2021a,Tang2009} and the rich references
therein). However, to the best of our knowledge,    
\cite{Gao2015} is the only published work addressing null controllability of  fourth order stochastic  parabolic equations, in which the authors establish the null controllability of \eqref{eq.forwardEquation} when $G$ is an interval.

We will use Carleman estimates to prove Theorem \ref{thm.observabilityForBackwardEquation}, which is deduced taking inspiration from the works \cite{Gao2015,Guerrero2019,Tang2009}.  
To state our Carleman estimate, we first introduce the weight functions:
\begin{equation}\label{eq.alpha}
	\alpha(x, t)=\frac{e^{\lambda(|2 \eta|_{C(\overline{G})}+\eta(x))}-e^{4 \lambda|\eta|_{C(\overline{G})}}}{t^{1/2}(T-t)^{1/2}}, \quad \xi(x, t)=\frac{e^{\lambda(|2 \eta|_{C(\overline{G})}+\eta(x))}}{t^{1/2}(T-t)^{1/2}},
\end{equation}
where $ \eta $ satisfies:
\begin{equation}\label{eq.eta}
	\eta \in C^{4}(\overline{G}),\quad  \eta >0 \text{ in } G, \quad \eta|_{{\partial G}}=0, \quad|\nabla \eta| \geq C>0 \text { in } \overline{ G \backslash G_{1}}
\end{equation}
with $ G_{1} \subset  \overline{G_{1}} \subset G_{0} $ an open set.
For the existence of $ \eta $, see \cite[Lemma 1.1]{Fursikov1996}.
We remark that our weight functions are the same as that in \cite{Guerrero2019}, except for a minus sign.

Let $ \theta=e^{\ell} $, $ \ell=s\alpha $. We have the following Carleman estimate  for \eqref{eq.adjointBackwardEquation}.

\begin{theorem}\label{thm.carlemanBackward}
	There exists a constant $C>0$  such that for all  $z_{T} \in L_{\mathcal{F}_{T}}^{2}(\Omega ; L^{2}(G))$, the weak solution $ (z(\cdot),Z (\cdot) )$ to \eqref{eq.adjointBackwardEquation} satisfies that
	\begin{align}
		\notag
		&\mathbb{E} \int_{Q}  (s^{6} \lambda^{8}  \xi^{6} \theta^{2}| z |^{2}+s^{4} \lambda^{6}  \xi^{4} \theta^{2}|\nabla z |^{2} +s^{3} \lambda^{4}  \xi^{3} \theta^{2}|\Delta z |^{2} + s^{2} \lambda^{4}  \xi^{2} \theta^{2}| \nabla^{2} z |^{2}   ) d x d t
		\\\label{eq.carlemanBackward}
		&      \leq
		C  \mathbb{E} \int_{Q_{0}} s^{7} \lambda^{8} \xi^{7}\theta^{2}| z |^{2} d x d t
		+C e^{C \lambda |\eta|_{C(\overline{G})}} \mathbb{E} \int_{Q} s^{4} \lambda^{4} \xi^{4} \theta^{2} | Z |^{2} dxdt
		+ C r_{1}^{2} \mathbb{E} \int_{Q}\theta^{2} ( | z |^{2} + |Z|^{2} ) d x d t
	\end{align}
	for $ \lambda \geq C $ and $ s\geq C(T^{1/2}+T) $.
\end{theorem}

\begin{remark}
Theorem \ref{thm.carlemanBackward} can be regarded as a generalization of \cite[Theorem 1.2]{Guerrero2019} for stochastic equations. Due to the appearance of the diffusion term, we cannot simply mimic the proof of \cite[Theorem 1.2]{Guerrero2019} to prove Theorem \ref{thm.carlemanBackward}.  
\end{remark}

\begin{remark}
	
	From the proof of Theorem \ref{thm.carlemanBackward}, it is not difficult to see that Carleman inequality \eqref{eq.carlemanBackward} is still valid when $ z $ satisfies the boundary condition
	\begin{equation}\label{eq.01BC}
		z=\frac{\partial z}{\partial \nu}=0 \  \text{on} \  \Sigma,
	\end{equation}
	where  $\nu(x)=(\nu^{1}(x), \cdots, \nu^{n}(x))$ stands for the unit outward normal vector (of $G$) at $x \in \partial G$. However, the situation for
	\begin{equation}\label{eq.otherBC}
		\frac{\partial z}{\partial \nu}=\frac{\partial \Delta z}{\partial \nu}=0 \  \text{on} \  \Sigma
	\end{equation}
	is more difficult.  As far as we know, it is still an open problem for the validity of Carleman estimates for deterministic fourth order parabolic operator under the boundary conditions \eqref{eq.otherBC} in the multi-dimensional case.
	
\end{remark}

The rest of this paper is organized as follows.
 In Section \ref{section.weightedIndentity}, we give some preliminaries.
The Carleman  estimate  \eqref{eq.carlemanBackward} is proved in Section \ref{section.carleman}. In Section \ref{section.ObserAndControl}, we prove Theorem \ref{thm.observabilityForBackwardEquation}.
Finally, we finished our article with appendices in which we prove the weighted identity and lemmas that are important in deriving the Carleman estimate.

\section{Some preliminary results}\label{section.weightedIndentity}

This section is devoted to presenting  some preliminary results for the proof of Theorem \ref{thm.carlemanBackward}. 

We first give
a pointwise weighted identity for the (formal) stochastic partial differential operator  $- dh+\Delta^{2} h dt$.

In what follows, for simplicity, we will use the notation $y_{x_{i}} \equiv y_{x_{i}}(x)=\partial y(x) / \partial x_{i}$, where $x_{i}$ is the $i$-th coordinate of a generic point $x=(x_{1}, \ldots, x_{n})$ in $\mathbb{R}^{n} $.

Recall that $ \theta=e^{\ell} $ and $ \ell=s\alpha $. Set $ w=\theta h $.
Put
\begin{equation}\label{eq.pfthmFI.3}
    I = \Delta^{2} w  + \Psi_{2}\Delta w + \sum_{i, j=1}^n  \Psi_{3} ^{ij} w_{x_{i}x_{j}} + \sum_{i=1}^n \Psi_{4}^{i}  w_{x_{i}}+ \sum_{i=1}^n \Psi_{5}^{i}  w_{x_{i}}  + \Psi_{6} w,\quad I_{1}= Idt,
\end{equation}
where for $ i,j=1,\cdots,n $,
\begin{align}
    \notag
    & \Psi_{2} = 2 s^{2}\lambda^{2}\xi^{2} |\nabla\eta|^{2},
    \quad
    \Psi_{3}^{ij} = 4 s^{2}\lambda^{2}\xi^{2} \eta_{x_{i}} \eta_{x_{j}},
    \quad
    \Psi_{4}^{i} = 12 s^{2}\lambda^{3}\xi^{2} |\nabla\eta|^{2} \eta_{x_{i}}  +4 s^{2} \lambda^{2} \xi^{2} \Delta \eta \eta_{x_{i}} ,
    \\ \label{eq.pfthmFI.3.1}
    &
    \Psi_{5}^{i} =  \sum_{j=1}^n 8 s^{2}\lambda^{2}\xi^{2} \eta_{x_{i}x_{j}} \eta_{x_{j}},
    \quad
    \Psi_{6} = s ^{4}\lambda ^{4}\xi^{4} |\nabla \eta|^{4}.
\end{align}
Let
\begin{equation}\label{eq.pfthmFI.4}
    I_{2}  =    \sum_{i=1}^n \Phi_{1}^{i} \Delta w_{x_{i}} dt + \Phi_{2} \Delta w dt + \sum_{i, j=1}^n \Phi_{3}^{ij}  w_{x_{i}x_{j}} dt + \sum_{i=1}^n \Phi_{4} ^{i}  w_{x_{i}} dt + \Phi_{5} w dt
        - dw,
\end{equation}
where for $ i,j=1,\cdots,n $,
\begin{align}
    \notag
& \Phi_{1}^{i} = -4 s \lambda \xi  \eta_{x_{i}},
\quad
\Phi_{2} =  -2 s\lambda^{2}\xi |\nabla \eta|^{2} - 2s\lambda\xi  \Delta \eta,
\quad
\Phi_{3}^{ij} = 4 s\lambda\xi  \eta_{x_{i}x_{j}} - 4 s\lambda^{2}\xi \eta_{x_{i}} \eta_{x_{j}},
\\ \label{eq.pfthmFI.4.1}
&
\Phi_{4}^{i} = -4 s^{3} \lambda^{3} \xi^{3} |\nabla\eta|^{2} \eta_{x_{i}} ,\quad \Phi_{5} =  - 6 s^{3} \lambda ^{4}\xi^{3} |\nabla\eta|^{4} - 12 s ^{3}\lambda ^{3}\xi^{3}  (\nabla^{2}\eta \nabla \eta \nabla \eta )  -2 s^{3} \lambda ^{3}\xi^{3} |\nabla\eta|^{2} \Delta \eta.
\end{align}
Put
\begin{align}
    \notag
    I_{3} & = -8 s \lambda \xi \sum\limits_{i,j=1}^{n} \eta_{x_{i}x_{j}} w_{x_{i}x_{j}}dt -4\nabla  \Delta \ell \cdot \nabla w  dt + 4  (\nabla \ell \cdot \nabla \Delta \ell ) w dt
        +2 |\nabla^{2} \ell|^{2} wdt
        \\\label{eq.pfthmFI.5}
        & \quad
        -\Delta^{2} \ell wdt+|\Delta \ell|^{2} wdt+ 8 s^{3} \lambda^{3} \xi^{3}(\nabla^{2} \eta \nabla \eta \nabla \eta) wdt
        + s \alpha_{t} w dt.
\end{align}

We have the following fundamental identity.

\begin{theorem}\label{thm.fundamentalIndentity}
    Let $ h $ be an $ H^{4}(G) $-valued It\^o process. Set $ \theta=e^{\ell} $, $ \ell=s\alpha $ and $ w=\theta h $, the function $ \alpha, \xi, $ and $ \eta $ being as in \eqref{eq.alpha} and \eqref{eq.eta}. Then, for any $ t\in[0,T] $ and a.e. $ (x,\omega) \in G \times \Omega $,
    \begin{align}
        \notag
        & 2 \theta I ( \!- dh \!+\! \Delta^{2} h dt )
        \!+\! 2  \sum_{i, j=1}^n (
        w_{x_{i}x_{i}x_{j}} dw \!-\! w_{x_{i}x_{j}}dw_{x_{i}}
        \!+\!\Psi_{2} w_{x_{i}} \delta_{ij} dw
        \!+\!\Psi_{3}^{ij} w_{x_{i}} dw
        )_{x_{j}}
        \!-\! 2 \operatorname{div} (V_{1}\!+\!V_{2}) dt
        \\ \label{eq.finalEquation}
        & =
        2 I^{2} dt
        + 2 II_{3}
        + 2 Mdt
        + 2 \sum_{i,j,k,p=1}^n \Lambda^{ijkp}_{1} w_{x_{i}x_{j}}w_{x_{k}x_{p}}dt
        + 2 \sum_{i, j=1}^n \Lambda_{2}^{ij} w_{x_{i}} w_{x_{j}}dt
        + 2 \Lambda_{3} w^{2}dt
        \\ \notag
        & \quad -  2 d \Big(
        \frac{1}{2} |\nabla^{2} w|^{2}
        -\frac{1}{2} \Psi_{2} |\nabla w |^{2}
        - 2 s^{2} \lambda^{2} \xi^{2} |\nabla \eta \nabla w|^{2}
        + \frac{1}{2} \Psi_{6} w^{2}
        \Big)
        - 4  s^{2}\lambda^{2}\xi^{2} \theta^{2} | \nabla \eta d \nabla h\! +\! \nabla \eta \nabla \ell dh|^{2}
        \\
        &\notag \quad
        + \theta^{2} \sum_{i, j=1}^n [  d h_{x_{i}x_{j}} + 2 \ell_{x_{j}} dh_{x_{i}} +   ( \ell_{x_{i}x_{j}} + \ell_{x_{i}} \ell_{x_{j}} ) dh ]^{2}
        -  \theta^{2}  \Psi_{2} | d \nabla h + \nabla \ell dh|^{2}
        +  \Psi_{6} \theta^{2} ( dh )^{2},
    \end{align}
    where
    \begin{align*}
        V_{1} & =[ V_{1}^{1}, \cdots,  V_{1}^{j},\cdots, V_{1}^{n} ],
        \quad
        V_{2} = [ V_{2}^{1}, \cdots,  V_{2}^{j},\cdots, V_{2}^{n} ],
        \\
        V_{1}^{j} & =  \sum_{i,k=1}^n\Big[ \sum_{p=1}^n \Phi_{1}^{p} w_{x_{k}x_{k}x_{p}} w_{x_{i}x_{i}x_{j}}  -  \frac{1}{2} \sum_{p=1}^n  \Phi_{1}^{j} w_{x_{k}x_{k}x_{p}} w_{x_{i}x_{i}x_{p}}
            + \frac{1}{2}  \Psi_{2} \Phi_{1}^{j}w_{x_{i}x_{i}}w_{x_{k}x_{k}}
          \\
          &  \quad \quad \quad \ \  
           + \sum_{p=1}^{n} \Psi_{3}^{ik}\Phi_{1}^{p} w_{x_{i}x_{k}}w_{x_{p}x_{j}}
           -  \frac{1}{2}  \sum_{p=1}^{n}\Psi_{3}^{ij}\Phi_{1}^{p} w_{x_{i}x_{k}}w_{x_{k}x_{p}}
            + \Phi_{4}^{k} w_{x_{i}x_{i}x_{j}} w_{x_{k}}
            - \Phi_{4}^{k} w_{x_{i}x_{j}} w_{x_{i}x_{k}}
            \\
            & \quad \quad \quad \ \  
            + \frac{1}{2} \Phi_{4}^{j}w_{x_{i}x_{k}}^{2}
            +\Big ( \Psi_{2} \Phi_{4}^{k} \delta_{ij}
            - \frac{1}{2} \Psi_{2} \Phi_{4}^{j} \delta_{ik}
            - \frac{1}{2} \Psi_{6} \Phi_{1}^{j} \delta_{ik}
            + \Psi_{3}^{ij}\Phi_{4}^{k} \Big ) w_{x_{i}}w_{x_{k}}\Big],
        \\
        V_{2}^{j}  &=
        \sum_{i,k,p,r,q=1}^n \Theta_{1}^{ijkprq} w_{x_{i}x_{k}x_{p}}w_{x_{r}x_{q}}
        + \sum_{i,k,p,r=1}^n \Theta_{2}^{ijkpr} w_{x_{i}x_{k}}w_{x_{p}x_{r}}
        + \sum_{i,k,p=1}^n \Theta_{3}^{ijkp} w_{x_{i}x_{k}}w_{x_{p}}
        \\
        & \quad
        + \sum_{i,k=1}^n \Theta_{4}^{ijk} w_{x_{i}}w_{x_{k}}
        + \sum_{i=1}^n \Theta_{5} w_{x_{i}x_{i}x_{j}}w
        + \sum_{i,k=1}^n \Theta_{6}^{ijk} w_{x_{i}x_{k}}w
        + \sum_{i=1}^n \Theta_{7}^{ij} w_{x_{i}}w
        +  \Theta_{8}^{j} w^{2},
    \end{align*}
    \begin{align*}
        M & =
        8 s \lambda^{2}\xi   | \nabla \Delta w \cdot \nabla \eta  | ^{2} 
        + 32 s^{3} \lambda^{4} \xi^{3}  | \nabla^{2} w  \nabla \eta \nabla \eta   |^{2} 
        + 48 s^{3} \lambda^{3} \xi^{3} \nabla^{2} \eta  (\nabla^{2} w \nabla \eta )  ( \nabla^{2} w \nabla \eta  ) 
        \\
        & \quad
        + 16 s^{3} \lambda^{3} \xi^{3}  ( \nabla^{2} w \nabla \eta \nabla \eta  ) \sum\limits_{i,j=1}^{n} \eta_{x_{i}x_{j}} w_{x_{i}x_{j}} 
        -16 s^{3} \lambda ^{4} \xi ^{3} |\nabla \eta|^{2}  | \nabla^{2} w \nabla \eta  | ^{2}       
        \\
        & \quad +6 s^{3} \lambda^{4} \xi ^{3} |\nabla \eta |^{4} | \nabla^{2} w|^{2} 
        + 4 s^{3} \lambda^{3} \xi ^{3}  ( \nabla^{2} \eta \nabla \eta \nabla \eta  ) |\nabla ^{2} w|^{2} 
        + 2  s^{3} \lambda^{3} \xi ^{3} |\nabla \eta|^{2} \Delta \eta |\nabla^{2} w|^{2}         
        \\
        & \quad+ 2 s^{3} \lambda^{4} \xi ^{3} |\nabla \eta |^{4} | \Delta w|^{2} 
        - 4 s^{3} \lambda^{3} \xi ^{3}  ( \nabla^{2} \eta \nabla \eta \nabla \eta  ) |\Delta w|^{2} 
        - 2  s^{3} \lambda^{3} \xi ^{3} |\nabla \eta|^{2} \Delta \eta |\Delta w|^{2}       
        \\
        & \quad+ 40 s^{5} \lambda^{6} \xi^{5} |\nabla \eta|^{4} | \nabla w \cdot \nabla \eta|^{2} 
        + 64 s^{5} \lambda ^{5} \xi^{5}  ( \nabla^{2} \eta \nabla \eta \nabla \eta  ) |\nabla w \cdot \nabla \eta|^{2}       
        \\
        & \quad- 16 s^{5} \lambda^{6} \xi^{5} |\nabla \eta|^{6} |\nabla  w|^{2} 
        + 8 s^{7} \lambda ^{8} \xi^{7} |\nabla \eta|^{8} w^{2} - \sum_{i,k,j=1}^n ( \Psi_{5}^{i} \Phi_{1}^{k} )_{x_{j}} w_{x_{i}} w_{x_{k}x_{j}},
    \end{align*}
    \begin{align*}
    & \Lambda_{1}^{ijkp} =
    \sum_{r=1}^{n} \Big (  \frac{1}{2} \Phi_{2x_{r}x_{r}} \delta_{ij}\delta_{kp}
    - \Phi_{3x_{r}x_{r}}^{kp} \delta_{ij}
    + \sum_{q=1}^{n}\frac{1}{2} \Phi_{3x_{r}x_{q}}^{rq} \delta_{ij}\delta_{kp}
    + 2\Phi_{3x_{i}x_{j}}^{kp}
    - 2 \Phi_{3x_{j}x_{r}}^{kr} \delta_{ip}
    + \Phi_{3x_{r}x_{r}}^{jk} \delta_{ip}\Big ),
    \\
   & \Lambda_{2}^{ij}
        =
    \begin{aligned}[t]
        \sum_{k=1}^n\!\Big [ & - \Phi_{4 x_{i}x_{k}x_{k}}^{j}
    \!+\! \frac{1}{2} \Phi_{4x_{i}x_{j}x_{k}}^{k}
    \!\!- 2\Phi_{5x_{i}x_{j}}
    \!\!+  (\Psi_{2}\Phi_{3}^{jk} )_{x_{i}x_{k}}
    \!\!- \frac{1}{2}  (\Psi_{2}\Phi_{3}^{ij} )_{x_{k}x_{k}}
    \!\!- \sum_{p=1}^{n} \frac{1}{2}   (\Psi_{2}\Phi_{3}^{kp} \delta_{ij} )_{x_{k}x_{p}}
    \\
    &
    +   (\Psi_{3}^{ik}\Phi_{2}  )_{x_{k}x_{j}}
    - \sum_{p=1}^{n} \frac{1}{2}  (\Psi_{3}^{kp}\Phi_{2} \delta_{ij} )_{x_{k}x_{p}}
    - \frac{1}{2}  (\Psi_{3}^{ij}\Phi_{2}  )_{x_{k}x_{k}}
    + \sum_{p=1}^{n}  (\Psi_{3}^{ik}\Phi_{3}^{jp} )_{x_{k}x_{p}}
    \\
    &
    - \sum_{p=1}^{n} \frac{1}{2}  (\Psi_{3}^{ij}\Phi_{3}^{kp} )_{x_{k}x_{p}}
    - \frac{1}{2}  (\Psi_{3}^{kp}\Phi_{3}^{ij} )_{x_{k}x_{p}}
    + \frac{1}{2}  (\Psi_{4}^{i}\Phi_{1}^{j} )_{x_{k}x_{k}}
    +  (\Psi_{4}^{i}\Phi_{2} )_{x_{j}}
    + \frac{1}{2}  (\Psi_{4}^{k}\Phi_{2}  \delta_{ij} )_{x_{k}}
    \\
    &
    -  (\Psi_{4}^{i}\Phi_{3}^{jk} )_{x_{k}}
    + \frac{1}{2}  (\Psi_{4}^{k}\Phi_{3}^{ij} )_{x_{k}}
    +  (\Psi_{5}^{i}\Phi_{2} )_{x_{j}}
    + \frac{1}{2}  (\Psi_{5}^{k}\Phi_{2}  \delta_{ij} )_{x_{k}}
    -  (\Psi_{5}^{i}\Phi_{3}^{jk} )_{x_{k}}
    + \frac{1}{2}  (\Psi_{5}^{k}\Phi_{3}^{ij} )_{x_{k}}
    \\
    &
    - \frac{1}{2}  \Psi_{2t} \delta_{ij}
    - 4  s^{2}\lambda^{2}\xi ( \xi \eta_{x_{j}} )_{t} \eta_{x_{i}}
    \Big ],
    \end{aligned}
    \\
      &   \Lambda_{3} =
        \sum_{i=1}^n \Big[  \sum_{j=1}^n \frac{1}{2} \Phi_{5x_{i}x_{i}x_{j}x_{j}}
        + \sum_{j=1}^n \frac{1}{2}   ( \Psi_{2} \Phi_{5}  )_{x_{j}x_{j}}
        + \sum_{j=1}^n \frac{1}{2} (\Psi_{3}^{ij} \Phi_{5})_{x_{i}x_{j}}
        - \frac{1}{2}   ( \Psi_{4} ^{i} \Phi_{5}  )_{x_{i}}
        - \frac{1}{2}   ( \Psi_{5} ^{i} \Phi_{5}  )_{x_{i}}
        \\
        &  \quad \quad \quad \quad
        - \sum_{j=1}^n \frac{1}{2} ( \Psi_{6} \Phi_{1}^{i}  )_{x_{i}x_{j}x_{j}}
        + \frac{1}{2}  ( \Psi_{6} \Phi_{2}  )_{x_{i}x_{i}}
        + \frac{1}{2}   \Psi_{6t}   \Big]
    \end{align*}
    and
    \begin{align*}
     &   \Theta_{1}^{ijkprq} =\Phi_{2} \delta_{ik}\delta_{pj}\delta_{rq} + \Phi_{3}^{rq}\delta_{ik}\delta_{pj}- \Phi_{3}^{ik}\delta_{rq}\delta_{pj}+\Phi_{3}^{ij}\delta_{kp}\delta_{rq},                                                                                        \\
     &   \Theta_{2}^{ijkpr} =
        -\frac{1}{2}\Phi_{2x_{j}} \delta_{ik}\delta_{pr}
        + \Phi_{3x_{j}}^{pr}\delta_{ik}
        -\frac{1}{2}\sum_{q=1}^{n} \Phi_{3x_{q}}^{jq}\delta_{ik}\delta_{pr}
        -2 \Phi_{3x_{i}}^{pr}\delta_{kj}
        + 2\Phi_{3x_{k}}^{rj}\delta_{ip}
        - \Phi_{3x_{j}}^{kr}\delta_{ip},
        \\
      &   \Theta_{3}^{ijkp} =\!
        - \Phi_{4x_{i}}^{p}\delta_{kj}\!
        - \Phi_{5}\delta_{ik}\delta_{pj}\!
        + \Psi_{2}\Phi_{3}^{ik}\delta_{pj}\!
        -\Psi_{2}\Phi_{3}^{ij}\delta_{kp}\!
        + \Psi_{3}\Phi_{2}\delta_{ik}\!
        -\Psi_{3}\Phi_{2}\delta_{ij}\!
        +\Psi_{3}^{ik}\Phi_{3}^{pj}\!
        -\Psi_{3}^{ij}\Phi_{3}^{kp}\!
        \\
        & \qquad\quad\;
        +\Psi_{4}^{p}\Phi_{1}^{k}\delta_{ij}
        +\Psi_{5}^{p}\Phi_{1}^{k}\delta_{ij},
        \\
       &   \Theta_{4}^{ijk}=
        \Phi_{4x_{i}x_{j}}^{k}\!
        - \frac{1}{2}\Phi_{4x_{i}x_{k}}^{j}\!
        + 2 \Phi_{5x_{i}}\delta_{kj}-\Phi_{5x_{j}}\delta_{ik}\!
        -  ( \Psi_{2}\Phi_{3}^{jk}  )_{x_{i}}\!
        +\frac{1}{2}  ( \Psi_{2}\Phi_{3}^{ik}  )_{x_{j}}\!
        + \sum_{p=1}^{n} \frac{1}{2}  ( \Psi_{2}\Phi_{3}^{jp}  )_{x_{p}} \delta_{ik}
        \\
        & \qquad\quad
        - \sum_{p=1}^{n}  ( \Psi_{3}^{ip}\Phi_{2}  )_{x_{p}}\delta_{kj}\!
        + \frac{1}{2} \sum_{p=1}^{n}  ( \Psi_{3}^{pj}\Phi_{2}  )_{x_{p}}\delta_{ik}\!
        + \frac{1}{2}   ( \Psi_{3}^{ik}\Phi_{2}  )_{x_{j}}\!
        - \sum_{p=1}^{n}  ( \Psi_{3}^{ij}\Phi_{3}^{kp}  )_{x_{p}}\!
        + \frac{1}{2}  \sum_{p=1}^{n}  ( \Psi_{3}^{ik}\Phi_{3}^{jp}  )_{x_{p}}
        \\
        & \qquad\quad
        +\! \frac{1}{2}  \sum_{p=1}^{n}  ( \Psi_{3}^{pj}\Phi_{3}^{ik}  )_{x_{p}}\!
        -\! \frac{1}{2}    ( \Psi_{4}^{i}\Phi_{1}^{k}  )_{x_{j}}\!
        +\!  \Psi_{4}^{i} \Phi_{2} \delta_{kj}\!
        -\! \frac{1}{2}\Psi_{4}^{j} \Phi_{2} \delta_{ik}\!
        +\! \Psi_{4}^{i}\Phi_{3}^{kj}\!
        -\!\frac{1}{2} \Psi_{4}^{j}\Phi_{3}^{ik}\!
        + \! \Psi_{5}^{i} \Phi_{2} \delta_{kj}\!
        \\
        & \qquad\quad
        -\! \frac{1}{2}\Psi_{5}^{j} \Phi_{2} \delta_{ik}
        + \Psi_{5}^{i}\Phi_{3}^{kj}
        -\frac{1}{2} \Psi_{5}^{j}\Phi_{3}^{ik},
    \\
      &   \Theta_{5}= \Phi_{5},
        \\
      &   \Theta_{6}^{ijk}= - \Phi_{5x_{j}}\delta_{ik}+ \Psi_{6}\Phi_{1}^{k}\delta_{ij},
        \\
      &   \Theta_{7}^{ij} = \sum_{k=1}^{n} \Phi_{5x_{k}x_{k}}\delta_{ij} + \Psi_{2} \Phi_{5} \delta_{ij} + \Psi_{3}^{ij}\Phi_{5} -  ( \Psi_{6}\Phi_{1}^{i}  )_{x_{j}} + \Psi_{6}\Phi_{2} \delta_{ij} + \Psi_{6}\Phi_{3}^{ij},
        \\
      &   \Theta_{8}^{j} = - \frac{1}{2} \sum_{i=1}^n \Phi_{5x_{i}x_{i}x_{j}}\! -\! \frac{1}{2} ( \Psi_{2}\Phi_{5}  )_{x_{j}}\! -\! \frac{1}{2} \sum_{i=1}^n  ( \Psi_{3}^{ij} \Phi_{5}  )_{x_{i}}\! +\! \frac{1}{2} \Psi_{4}^{j}\Phi_{5}\! +\! \frac{1}{2} \Psi_{5}^{j}\Phi_{5}
      \!+\!\frac{1}{2}\sum_{i=1}^n  ( \Psi_{6}\Phi_{1}^{j}  )_{x_{i}x_{i}}
      \\
      & \quad \quad \
        -\!\frac{1}{2}    (\Psi_{6}\Phi_{2} )_{x_{j}}
        - \frac{1}{2}\sum_{i=1}^n  ( \Psi_{6}\Phi_{3}^{ij}  )_{x_{i}}
        +\frac{1}{2} \Psi_{6} \Phi_{4}^{j}.
    \end{align*}
\end{theorem}

\begin{remark}
    The weighted identity \eqref{eq.finalEquation} can be used to derive Carleman estimates for different boundary conditions, such as   \eqref{eq.01BC}  and \eqref{eq.otherBC}. Although the form of the identity is complex, many terms can merge or vanish under different boundary conditions.
\end{remark}

Proof of Theorem \ref{thm.fundamentalIndentity} is put in Appendix \ref{App-id}.

In what follows, for a nonnegative integer $m$, we denote by $O(\lambda^{m})$ a function of order $\lambda^{m}$ for large $\lambda$ (which is independent of $s$).
Likewise, we use the notation $O\bigl(e^{\lambda|\eta|_{C(\bar{G})}}\bigr)$ and so on.
It is easy to check that
\begin{equation}\label{eq.estimatesOrder}
    |\xi_{t}|  \leq \frac{T}{2} \xi^{3}, \quad |\alpha_{t}|  \leq \frac{T}{2} \xi^{3}, \quad \xi^{-1} \leq \frac{T}{2}.
\end{equation}
\begin{proposition}\label{prop.estimatesOrder}
    When $ s $ and $ \lambda $ are large enough, it holds that
    \begin{equation}\label{eq.estimatesOrderInequalities}
        \begin{cases}
            \begin{alignedat}{2}
                &\sum_{i,j,k,p=1}^{n} \Lambda^{ijkp}_{1}  \zeta _{ij} \zeta_{kp} \geq  s  \xi O(\lambda^{4}) |\zeta|^{2}, && \quad \forall \zeta=( \zeta_{ij} ) \in \mathbb{R}^{n\times n},
                \\
                &\sum_{i, j=1}^n \Lambda_{2}^{ij} \zeta_{i} \zeta_{j} \geq \big[s^{3}\xi^{3}O(\lambda^{6}) + Ts^{2}\xi^{4}O(\lambda^{2}) \big] |\zeta|^{2}  , && \quad \forall \zeta= ( \zeta_{1},\dots,\zeta_{n} ) \in \mathbb{R}^{n},
                \\
                &\Lambda_{3} \geq s ^{5} \xi^{5} O(\lambda^{8})  + Ts^{4}\xi^{6}O(\lambda^{4}) . &&
            \end{alignedat}
        \end{cases}
    \end{equation}
\end{proposition}

The following is the Carleman inequality for the Laplace operator with homogeneous Dirichlet boundary conditions.

\begin{lemma}\label{lemma.LaplaceCarleman}\cite[Lemma 2.4]{Guerrero2019}
	Let $ q \in H^{2}(G)\cap H_0^1(G) $ and $ \eta $ be given by \eqref{eq.eta}. Then there exists a constant $\mathrm{C}>0$ independent of $\mathrm{s}$ and $\lambda$, and parameters $\widehat{\lambda}>1$ and $\widehat{\mathrm{s}}>1$ such that for all $\lambda \geq \widehat{\lambda}$ and for all $\tilde{s} \geq \widehat{s}$
	\begin{align}
		\notag
		&\tilde{s}^{6} \lambda^{8} \int_{G} e^{6 \lambda \eta(x)} e^{2 \tilde{s} e^{\lambda \eta(x)}}|q|^{2} d x +\tilde{s}^{4} \lambda^{6} \int_{G} e^{4 \lambda \eta(x)} e^{2 \tilde{s} e^{\lambda \eta(x)}}|\nabla q|^{2} d x
		\\
		\label{eq.LaplaceCarleman}
		&
		\leq C \Big (\tilde{s}^{6} \lambda^{8} \int_{G_{1}} e^{6 \lambda \eta(x)} e^{2 \tilde{s} e^{\lambda \eta(x)}}|q|^{2} d x +\tilde{s}^{3} \lambda^{4} \int_{G} e^{3 \lambda \eta(x)} e^{2 \tilde{s} e^{2 \tilde{s} \lambda \eta(x)}}|\triangle q|^{2} d x \Big).
	\end{align}
\end{lemma}

\begin{remark}\label{remark.LaplaceCarleman}
	By taking $\tilde{s} =s \frac{\exp ( \lambda |2 \eta|_{C(\overline{G})})}{t^{1/2}(T-t)^{1/2}}$ in \eqref{eq.LaplaceCarleman},  multiplying \eqref{eq.LaplaceCarleman} by
	$ \operatorname{exp} \Big(2 s \frac{- \exp (4 \lambda |\eta|_{C(\overline{G})} )}{t^{1/2}(T-t)^{1/2}}\Big) $,
	and integrating it in $(0, T)$, we obtain that for $\lambda \geq \hat{\lambda}$ and $s \geq \hat{s} T$,
	\begin{equation*}\label{eq.ELaplaceCarleman}
		\int_{Q}\!\! s^{6} \lambda^{8} \xi^{6}\theta^{2} |q|^{2} d x d t\!+\!  \int_{Q}\!\!  s^{4} \lambda^{6}  \xi^{4}\theta^{2} |\nabla q|^{2} d x d t
		\leq  C  \Big(  \int_{(0, T)\times G_{1}}\!\!\! s^{6} \lambda^{8} \xi^{6}\theta^{2} |q|^{2} d x d t+   \int_{Q}\!\!  s^{3} \lambda^{4}  \xi^{3}\theta^{2} |\Delta q|^{2} d x d t \Big
		).
	\end{equation*}
\end{remark}

Denote
\begin{equation}
    \label{eq.alphastar}
\alpha_{\star}(t)=\frac{e^{\lambda|2 \eta|_{C(\overline{G})}}-e^{4 \lambda|\eta|_{C(\overline{G})}}}{t^{1/2}(T-t)^{1/2}}, \quad \quad
\xi_{\star}(t)=\frac{e^{\lambda|2 \eta|_{C(\overline{G})}}}{t^{1/2}(T-t)^{1/2}}.
\end{equation}

\begin{lemma}\label{lemma.auxiliaryEstimates}
	There exists a constant $ C>0 $, such that for all
	$ h \in  L_{\mathbb{F}}^{2}(0, T ; H^{4}(G)) , f, g \in L_{\mathbb{F}}^{2}(0, T ; L^{2}(G))$, satisfying
	\begin{equation*}
		\left\{
		\begin{alignedat}{2}
			& - d h + \Delta^2 h d t=    f d t+g d W(t) && \quad \text { in } Q, \\
			& h=\Delta h=0 && \quad \text { on } \Sigma,
		\end{alignedat}
		\right.
	\end{equation*}
	we have
	\begin{align*}
		\notag
		&\mathbb{E} \int_{\Sigma} \Big[
		s^{\frac{9}{4}} \lambda^{3} \xi^{\frac{9}{4}} \theta^{2} |\nabla^{2} h|^{2}
		+ s^{\frac{3}{4}} \lambda     \xi^{\frac{3}{4}} \theta^{2} |\nabla^{3} h|^{2}
		\Big]d\sigma dt
		+ \mathbb{E} \int_{0}^{T} e^{2s\alpha_{\star}} |h|^{2}_{H^{4}(G)} dt
		\\
		&
		\leq
		C \mathbb{E} \int_{Q}  s^{6} \lambda^{8} \xi^{6} \theta^{2} |h|^{2} dxdt
		+ C  \mathbb{E} \int_{Q}   \theta^{2} |f|^{2} dxdt,\qquad \forall s\geq C T^{1/2}.
	\end{align*}
\end{lemma}

Lemma \ref{lemma.auxiliaryEstimates} is a generalization of Lemma 2.3  in \cite{Guerrero2019} to stochastic equations. Its proof is put in Appendix \ref{sec.lemmas}.

\section{Proof of the Carleman estimate for the backward equation}\label{section.carleman}

\begin{proof}[Proof of Theorem \ref{thm.carlemanBackward}]

    Let
    $ h \in  L_{\mathbb{F}}^{2}(0, T ; H^{4}(G)),   f \in L_{\mathbb{F}}^{2}(0, T ; L^{2}(G)),  g \in L_{\mathbb{F}}^{2}(0, T ; H^{2}(G)) $
    satisfying
    \begin{equation}\label{eq.pfthm2.1.S1}
        \left\{
        \begin{alignedat}{2}
            &-d h+\Delta^2 h d t=f d t+g d W(t) && \quad \text { in } Q, \\
            &h=\Delta h=0 && \quad \text { on } \Sigma.
        \end{alignedat}
        \right.
    \end{equation}

    In order to simplify the formulae used in the sequel, we define
    \begin{align}
        \notag
        &\mathcal{A}_{1} = s^{5}  \mathbb{E}\int_{Q} \xi^{5}[O(\lambda^{8})+s  \xi O(\lambda^{7})]|w|^{2} d x d t ,
        & &
         \mathcal{A}_{2} = s^{3}  \mathbb{E}\int_{Q} \xi^{3}[O( \lambda^{6} )+s \xi O( \lambda^{5} )]|\nabla w|^{2} d x d t,
        \\ \label{eq.pfthm2.1.S8}
        &\mathcal{A}_{3} = s  \mathbb{E}\int_{Q} \xi[O( \lambda^{4} )+s \xi O( \lambda^{3} )]|\nabla^{2} w|^{2} d x d t,
        & &
        \mathcal{B}_{1} = s^{4}  \mathbb{E}\int_{\Sigma} \xi^{4}[ O ( \lambda^{5} ) +s \xi
         O ( \lambda^{4} ) ]\Big(\frac{\partial w}{\partial \nu}\Big)^{2} d \sigma d t,\\ \notag
        & \mathcal{B}_{2} =  s^{2} \mathbb{E} \int_{\Sigma} \xi^{2} O ( \lambda^{3} ) \theta^{2}|\nabla^{2} h|^{2} d \sigma d t,
        & &
        \mathcal{B}_{3} =  s^{1/2}  \mathbb{E} \int_{\Sigma} \xi^{1/2} O ( \lambda ) \theta^{2}|\nabla^{3} h|^{2} d \sigma d t,
    \end{align}
    and
    \begin{equation*}
        \mathcal{A}=\mathcal{A}_{1}+\mathcal{A}_{2}+\mathcal{A}_{3},
        \quad \quad
        \mathcal{B}=\mathcal{B}_{1}+\mathcal{B}_{2}+\mathcal{B}_{3}.
    \end{equation*}

    Integrating the equality \eqref{eq.finalEquation} on $Q$, taking mathematical expectation in both sides, and noting \eqref{eq.estimatesOrder}, \eqref{eq.estimatesOrderInequalities}, we conclude that, for $ s \geq C T^{1/2} $
    \begin{align}
        \notag
        & 2 \mathbb{E}\int_{Q} \theta I ( - dh + \Delta^{2} h dt ) dx
         + 2 \mathbb{E}\int_{Q} \sum_{i, j=1}^n [
       w_{x_{i}x_{i}x_{j}} dw - w_{x_{i}x_{j}}dw_{x_{i}}
       +\Psi_{2} w_{x_{i}} \delta_{ij} dw
       +\Psi_{3}^{ij} w_{x_{i}} dw
       ]_{x_{j}} dx
       \\ \label{eq.thm31S1}
        &
        - 2  \mathbb{E}\int_{Q} \operatorname{div} V_{1} dxdt
       - 2 \mathbb{E}\int_{Q} \operatorname{div} V_{2} dx dt
       \\ \notag
        & \geq
       2 \mathbb{E}\!\int_{Q}\! I^{2} dxdt
       \!+\! 2 \mathbb{E}\!\int_{Q}\! II_{3} dx
       \!+\! 2 \mathbb{E}\!\int_{Q}\! M dxdt
       \!+\! \mathbb{E}\int_{Q} \theta^{2}\! \sum_{i, j=1}^n\!\big[  d h_{x_{i}x_{j}}\! +\! 2 \ell_{x_{j}} dh_{x_{i}}\! +\!   ( \ell_{x_{i}x_{j}} + \ell_{x_{i}} \ell_{x_{j}} ) dh \big]^{2} dx
        \\ \notag
        & \quad
        \!-\! \mathbb{E}\int_{Q} \theta^{2}  \Psi_{2} | d \nabla h
        \!+\! \nabla \ell dh|^{2} dx
        \!-\! 4 \mathbb{E}\int_{Q}  s^{2}\lambda^{2}\xi^{2} \theta^{2} | \nabla \eta d \nabla h
        \!+\! \nabla \eta \nabla \ell dh|^{2} dx
        \!+\!  \mathbb{E}\int_{Q}  \Psi_{6} \theta^{2} ( dh )^{2} dx + \mathcal{A}
       .
    \end{align}

    We will estimate the divergence terms, the interior terms, the stochastic terms, and the boundary terms in the following Steps 1--4.

    {\bf Step 1}. In this step, we deal with the divergence term.

    Using the boundary conditions satisfied by $h$, on the boundary $ \Sigma $, we get that
    \begin{equation}\label{eq.thm31S2}
        w=0,\quad \nabla w = \frac{\partial w}{\partial \nu} \nu ,\quad \Delta w=2 s \lambda \xi \frac{\partial \eta}{\partial \nu} \frac{\partial w}{\partial \nu} ,
        \quad \frac{\partial^{2} w}{\partial \nu^{2}}=2 s \lambda \xi \frac{\partial \eta}{\partial \nu} \frac{\partial w}{\partial \nu} -\operatorname{div}(\nu)\frac{\partial w}{\partial \nu} 
    \end{equation}
and
\begin{equation}\label{eq.thm31BoundaryEstimates1}
\begin{aligned}
&|\nabla^{2} w|^{2}=
         s^{2} \xi^{2} O(\lambda^{2}) \Big( \frac{\partial w}{\partial \nu} \Big)^{2}
        + O(1) \theta ^{2} | \nabla^{2} h|^{2},
            \\
            &|\nabla^{3} w|^{2}=
          s^{4} \xi^{4} O(\lambda^{4}) \Big ( \frac{\partial w}{\partial \nu} \Big)^{2}
        + s^{2} \xi^{2} O(\lambda^{2}) \theta ^{2} | \nabla^{2} h|^{2}
        + O(1) \theta ^{2} | \nabla^{3} h|^{2},
        \end{aligned} \quad \mbox{for } s \geq C T.
    \end{equation}

Combining \eqref{eq.estimatesOrder} and \eqref{eq.thm31S2},  for $ s\geq C T $, we obtain that
\begin{align}\label{eq.thm31S4}
    &   -2  \mathbb{E}  \int_{Q} \operatorname{div} V_{1} dxdt\nonumber
    \\
    &   \leq
    \!- \!2
     \mathbb{E}\! \int_{\Sigma}\!  \Big[
    \!-\! 4 s \lambda  \xi \frac{\partial \eta}{\partial \nu} \!\Big(  \frac{\partial \Delta w}{\partial \nu}  \Big)^{2}\!
    \!+\! 2 s \lambda  \xi \frac{\partial \eta}{\partial \nu} | \nabla \Delta w |^{2}
    \!\!-\! 90 s^{5} \lambda^{5}  \xi^{5} \!\Big( \frac{\partial \eta}{\partial \nu} \Big) ^{5} \!\Big(  \frac{\partial w}{\partial \nu}  \Big)^{2}
    \!+\! 12 s^{3} \lambda^{3}  \xi^{3}  \!\Big( \frac{\partial \eta}{\partial \nu} \Big) ^{3} | \nabla^{2} w \!\cdot\! \nu |^{2}\nonumber
    \\
    &\qquad\qquad \ \
    - 4 s^{3} \lambda^{3}  \xi^{3}  \Big( \frac{\partial \eta}{\partial \nu} \Big) ^{3}  \frac{\partial w}{\partial \nu}   \frac{\partial \Delta w}{\partial \nu}
    - 2 s^{3} \lambda^{3}  \xi^{3}  \Big( \frac{\partial \eta}{\partial \nu} \Big) ^{3} | \nabla^{2} w |^{2}  \Big] d\sigma dt
    +\mathcal{B}_{1}.
\end{align}

Combining \eqref{eq.pfthmFI.3.1}, \eqref{eq.pfthmFI.4.1}, \eqref{eq.estimatesOrder}, \eqref{eq.thm31S2} and \eqref{eq.thm31BoundaryEstimates1}, for $ s\geq C T $, we have
\begin{align}
    \label{eq.thm31S5}
        - 2 \mathbb{E}  \int_{Q} \operatorname{div} V_{2} dx dt
            \leq \mathcal{B}_{1}+ \mathcal{B}_{2} + \mathcal{B}_{3}.
\end{align}

    {\bf Step 2.} In this step, we simplify $ 2 \mathbb{E}\int_{Q} M dxdt $  via integration by parts.

    Thanks to \eqref{eq.pfthmFI.3.1}, \eqref{eq.pfthmFI.4.1} and  \eqref{eq.pfthm2.1.S8}, we obtain
    \begin{equation}\label{eq.MK0}
        2 \mathbb{E} \int_{Q} - \sum_{i,k,j} ( \Psi_{5}^{i} \Phi_{1}^{k} )_{x_{j}} w_{x_{i}} w_{x_{k}x_{j}}   dxdt \geq - (\mathcal{A}_{2}+\mathcal{A}_{3}).
    \end{equation}
    Combining \eqref{eq.estimatesOrder}, \eqref{eq.pfthm2.1.S8} and \eqref{eq.thm31S2}, using integration by parts, for $ s\geq CT $, we have
    \begin{align}
        \notag &
         - \mathbb{E}\int_{Q}  32 s^{3} \lambda ^{4} \xi ^{3} |\nabla \eta|^{2}  | \nabla^{2} w \nabla \eta  | ^{2} dxdt
       \\ &\geq
        -8 \mathbb{E}\int_{Q}  s \lambda^{2} \xi |\nabla \Delta w \nabla \eta|^{2} dxdt
        -32 \mathbb{E}\int_{Q}  s^{5} \lambda^{6} \xi^{5} |\nabla \eta|^{4} |\nabla  w \nabla \eta|^{2} dxdt
        - (\mathcal{A}_{2}+\mathcal{B}_{1}). \label{eq.MK1}
    \end{align}

    Thanks to \eqref{eq.eta}, we deduce that for $\lambda \geq C$,
    \begin{align}
        \notag
        &2 \mathbb{E}  \int_{Q}  64 s^{5} \lambda ^{5} \xi^{5}  ( \nabla^{2} \eta \nabla \eta \nabla \eta  ) |\nabla w \cdot \nabla \eta|^{2} dxdt
        \\\label{eq.MK2}
        &
        \geq  - C  \mathbb{E}  \int_{(0, T) \times G_{1}}s^{5} \lambda ^{5} \xi^{5} |\nabla w|^{2} dxdt -  \mathbb{E}  \int_{Q} s^{5} \lambda^{6} \xi^{5} |\nabla \eta|^{4} | \nabla w \cdot \nabla \eta|^{2} dxdt.
    \end{align}

From \eqref{eq.estimatesOrder} and  \eqref{eq.pfthm2.1.S8}, we see that for $ s\geq CT $,
    \begin{align}
        2 \mathbb{E}  \int_{Q} ( 4 s^{3} \lambda^{3} \xi ^{3}  ( \nabla^{2} \eta \nabla \eta \nabla \eta  ) +2  s^{3} \lambda^{3} \xi ^{3} |\nabla \eta|^{2} \Delta \eta) ( |\nabla ^{2} w|^{2}-|\Delta w|^{2} )
        \geq - (\mathcal{A}_{2}+\mathcal{B}_{1}).
    \end{align}

    Using the argument similar to \eqref{eq.MK1}, we get that for $ \lambda \geq C $ and $ s\geq CT $,
    \begin{align}
        \notag
        &\mathbb{E}  \int_{Q}
        32 s^{3} \lambda^{3} \xi^{3}  ( \nabla^{2} w \nabla \eta \nabla \eta  ) \sum\limits_{i,j=1}^{n} \eta_{x_{i}x_{j}} w_{x_{i}x_{j}} dxdt
        \\
        &
        \geq  \mathbb{E} \int_{Q} 32  s^{3} \lambda^{3} \xi^{3} \nabla^{2} \eta  (\nabla^{2} w \nabla \eta )  ( \nabla^{2} w \nabla \eta  ) dxdt - ( \mathcal{A}_{2}+\mathcal{A}_{3}+\mathcal{B}_{1}+\mathcal{B}_{2} )
    \end{align}
    and
    \begin{align}
        \notag
        & 2 \mathbb{E}  \int_{Q} 64   s^{3} \lambda^{3} \xi^{3} \nabla^{2} \eta   (\nabla^{2} w \nabla \eta )   ( \nabla^{2} w \nabla \eta  ) dxdt
        \\ \notag
        &\geq  - C \sum_{i,j,k=1}^{n} \mathbb{E}  \int_{Q}  s^{3} \lambda^{3} \xi^{3} \eta_{x_{i}}\eta_{x_{j}} w_{x_{i}x_{k}}w_{x_{j}x_{k}}  dxdt
            \\\label{eq.MK4p}
            & \geq - C \mathbb{E}\int_{Q}   s^{3} \lambda ^{3} \xi ^{3} \sum_{i,j,k=1}^{n}  \eta_{x_{i}}^{2} \eta_{x_{i}} \eta_{x_{j}} w_{x_{i}x_{k}x_{k}} w_{x_{j}} dxdt -( \mathcal{A}_{2} + \mathcal{B}_{1} )
            \\\notag
            &\geq - \mathbb{E}\int_{Q}  s \lambda^{2} \xi |\nabla \Delta w \nabla \eta|^{2} dxdt
            - \mathbb{E}\int_{Q}  s^{5} \lambda^{6} \xi^{5} |\nabla \eta|^{4} |\nabla  w \nabla \eta|^{2} dxdt
            -( \mathcal{A}_{2} + \mathcal{B}_{1} )
    \end{align}

     {\bf Step 3}. In this step, we estimate  the stochastic terms.

    Using Divergence Theorem and noting that $ w=0 $ on $ \Sigma $, we have
    \begin{equation}\label{eq.thm31Sto1}
        2 \mathbb{E}\int_{Q} \sum_{i, j=1}^n \big(
            w_{x_{i}x_{i}x_{j}} dw
            +\Psi_{2} w_{x_{i}} \delta_{ij} dw
            +\Psi_{3}^{ij} w_{x_{i}} dw
            \big)_{x_{j}} dx
            = 0.
    \end{equation}

    Thanks to \eqref{eq.thm31S2}, noting that
    \begin{equation*}
        \lim _{t \rightarrow 0^{+}} \theta(t, \cdot)=\lim _{t \rightarrow T^{-}} \theta(t, \cdot) \equiv 0,
    \end{equation*}
we deduce that for $ s\geq C T $,
\begin{eqnarray}\label{eq.thm31Sto2}
&&- 2 \mathbb{E}\int_{Q} \sum_{i, j=1}^n (
w_{x_{i}x_{j}}dw_{x_{i}}
)_{x_{j}} dx \nonumber\\
&&=-2\mathbb{E} \int_{\Sigma}  \Big( 2 s \lambda \xi \frac{\partial \eta}{\partial \nu}  -\operatorname{div}(\nu)  \Big) \frac{\partial w}{\partial \nu} d \frac{\partial w}{\partial \nu} d\sigma  \nonumber  \\
&& =-\mathbb{E}\Big\{\! \int_{\Sigma} d \Big[  \Big( 2 s \lambda \xi \frac{\partial \eta}{\partial \nu}\!  -\operatorname{div}(\nu)  \Big)  \Big(\frac{\partial w}{\partial \nu} \Big)^{2}  \Big] d\sigma + \int_{\Sigma} \Big( 2 s \lambda \xi \frac{\partial \eta}{\partial \nu} -\operatorname{div}(\nu)\Big)_{t}  \Big(\frac{\partial w}{\partial \nu} \Big)^{2}  d\sigma dt  \\
&& \qquad\quad + \int_{\Sigma}  \Big( 2 s \lambda \xi \frac{\partial \eta}{\partial \nu}-\operatorname{div}(\nu)\Big)  \Big(d \frac{\partial w}{\partial \nu} \Big)^{2}  d\sigma\Big\}
\nonumber  \\
&&\leq \mathbb{E} \int_{\Sigma}    s^{3}  \xi^{4} O( \lambda )  \Big(\frac{\partial w}{\partial \nu} \Big)^{2}  d\sigma dt \leq   \mathcal{B}_{1}.\nonumber
\end{eqnarray}

    From \eqref{eq.pfthm2.1.S1}, and noting that $ \frac{1}{3} a^{2} \leq ( a+b+c )^{2} + b^{2}+c^{2} $ for any $a,b,c\in\dbR$, it follows that
\begin{align}
\notag &\mathbb{E}\int_{Q} \theta^{2} \sum_{i,
j=1}^n [  d h_{x_{i}x_{j}} + 2 \ell_{x_{j}}
dh_{x_{i}} +   ( \ell_{x_{i}x_{j}} +
\ell_{x_{i}} \ell_{x_{j}} ) dh ]^{2} dx dt
\\\label{eq.pfthm2.1.S13}
&\geq \frac{1}{3} \mathbb{E} \int_{Q} \theta^{2}
|\nabla^{2}g|^{2} dx - C \mathbb{E} \int_{Q}
s^{2} \lambda^{2} \xi^{2} \theta^{2} |\nabla
g|^{2} dxdt -C \mathbb{E} \int_{Q} s^{4}
\lambda^{4} \xi^{4} \theta^{2} |g|^{2} dxdt.
\end{align}

From \eqref{eq.pfthm2.1.S1} again,  we get that
    \begin{align}
        \notag
        &- \mathbb{E}\int_{Q} \theta^{2}  \Psi_{2} | d \nabla h + \nabla \ell dh|^{2} dx- 4 \mathbb{E}\int_{Q}  s^{2}\lambda^{2}\xi^{2} \theta^{2} | \nabla \eta d \nabla h + \nabla \eta \nabla \ell dh|^{2} dx
        \\\label{eq.thm31Sto3}
        &
        \geq - C \mathbb{E} \int_{Q} s^{2} \lambda^{2} \xi^{2} \theta^{2} |\nabla g|^{2} dxdt
        -C \mathbb{E} \int_{Q} s^{4} \lambda^{4} \xi^{4} \theta^{2} |g|^{2} dxdt.
    \end{align}

    For any $ t\in (0,T) $ and $ \varepsilon>0 $, we use the interpolation inequality (e.g., \cite[Theorem
    9.1]{Lions1972a}) and Young's inequality to obtain
    \begin{equation*}
        \int_{G}|\nabla(\theta g)|^{2}dx \leq \varepsilon \int_{G}|\nabla^{2}(\theta g)|^{2}dx+\frac{C}{\varepsilon} \int_{G} \theta^{2}|g|^{2}dx.
    \end{equation*}
  Take
    \begin{equation*}
        \varepsilon = \varepsilon_{2} \Big(\frac{s \lambda}{t^{1/2}(T-t)^{1/2}}\Big)^{-2},
    \end{equation*}
    where $ \varepsilon_{2} $ will be fixed later. Using
    \begin{equation*}\label{eq.pfthm2.1.S141}
        [t^{1/2}(T-t)^{1/2}]^{-1}\leq \xi \leq e^{3 \lambda |\eta|_{C(\overline{G})}}[t^{1/2}(T-t)^{1/2}]^{-1},
    \end{equation*}
    we obtain
    \begin{align*}
        \notag
         \int_{G} \theta^{2} |\nabla g|^{2}dx
         & \leq \varepsilon \int_{G}|\nabla^{2}(\theta g)|^{2}dx+\frac{C}{\varepsilon} \int_{G} \theta^{2}|g|^{2}dx+2\int_{G}|\nabla \theta|^{2} g^{2}dx
        \\
        & \leq \varepsilon_{2} s^{-2} \lambda^{-2} \!\int_{G}\! t(T\!-\!t)|\nabla^{2}(\theta g)|^{2}dx
        +C \varepsilon_{2}^{-1} s^{2} \lambda^{2} e^{C \lambda |\eta|_{C(\overline{G})}} \!\int_{G} t^{-1}(T\!-\!t)^{-1} \theta^{2} g^{2}dx,
    \end{align*}
    from which it follows that for $ s\geq CT $,
    \begin{align*}
        \int_{G} \theta^{2} |\nabla g|^{2}dx & \leq \varepsilon_{2} s^{-2} \lambda^{-2} \int_{G} t(T-t)\theta^{2}  |\nabla^{2} g|^{2}dx
        +C \varepsilon_{2}^{-1} s^{2} \lambda^{2} e^{C \lambda |\eta|_{C(\overline{G})}} \int_{G} t^{-1}(T-t)^{-1} \theta^{2} g^{2}dx.
    \end{align*}
 Hence,
    \begin{equation}\label{eq.pfthm2.1.S17}
        \int_{G} s^{2} \lambda^{2} \xi^{2} \theta^{2} |\nabla g|^{2}dx  \leq \varepsilon_{2} e^{C \lambda |\eta|_{C(\overline{G})}}  \int_{G} \theta^{2}  |\nabla^{2} g|^{2}dx
        +C \varepsilon_{2}^{-1} s^{4} \lambda^{4} e^{C \lambda |\eta|_{C(\overline{G})}} \int_{G} \xi^{4} \theta^{2} g^{2}dx.
    \end{equation}
    From \eqref{eq.pfthm2.1.S13}, \eqref{eq.thm31Sto3}, and \eqref{eq.pfthm2.1.S17},  letting $ \varepsilon_{2}>0 $ sufficiently small, it follows that     for $ s\geq CT $,
    \begin{align}
        \notag
        &\mathbb{E}\int_{Q} \theta^{2} \sum_{i, j=1}^n [  d h_{x_{i}x_{j}} + 2 \ell_{x_{j}} dh_{x_{i}} +   ( \ell_{x_{i}x_{j}} + \ell_{x_{i}} \ell_{x_{j}} ) dh ]^{2} dx
            - \mathbb{E}\int_{Q} \theta^{2}  \Psi_{2} | d \nabla h + \nabla \ell dh|^{2} dx
            \\\label{eq.pfthm2.1.S18}
             &- 4 \mathbb{E}\int_{Q}  s^{2}\lambda^{2}\xi^{2} \theta^{2} | \nabla \eta d \nabla h + \nabla \eta \nabla \ell dh|^{2} dx
            +  \mathbb{E}\int_{Q}  \Psi_{6} \theta^{2} ( dh )^{2} dx
            \\
            \notag
            & \geq
        -C e^{C \lambda |\eta|_{C(\overline{G})}} \mathbb{E} \int_{Q} s^{4} \lambda^{4} \xi^{4} \theta^{2} |g|^{2} dxdt.
    \end{align}

Combining \eqref{eq.thm31S1},
\eqref{eq.thm31S4}--\eqref{eq.thm31Sto2}, and
\eqref{eq.pfthm2.1.S18},     for $ s\geq CT $,
we obtain that
\begin{align}
\notag & 2 \mathbb{E}\int_{Q} \theta I ( - dh +
\Delta^{2} h dt ) dx
\\ \notag
& \geq 2 \mathbb{E}\int_{Q} I^{2} dxdt \!+\! 2
\mathbb{E}\int_{Q} II_{3} dx \!-\! C  \mathbb{E}
\int_{(0, T) \times G_{1}}s^{5} \lambda ^{5}
\xi^{5} |\nabla w|^{2} dxdt \!-\!C e^{C \lambda
|\eta|_{C(\overline{G})}} \mathbb{E} \int_{Q}
s^{4} \lambda^{4} \xi^{4} \theta^{2} |g|^{2}
dxdt
\\ \label{eq.thm31S6}
& \quad \!\! +  2 \mathbb{E}\int_{Q} [  3 s
\lambda^{2}\xi | \nabla \Delta w \cdot \nabla
\eta  | ^{2} + 32 s^{3} \lambda^{4} \xi^{3}  |
\nabla^{2} w \nabla \eta \nabla \eta |^{2} + 8
s^{3} \lambda^{4} \xi ^{3} |\nabla \eta |^{4} |
\Delta w|^{2}
\\  \notag
& \qquad\qquad + 22 s^{5} \lambda^{6} \xi^{5}
|\nabla \eta|^{4} | \nabla w \cdot \nabla
\eta|^{2} - 16 s^{5} \lambda^{6} \xi^{5} |\nabla
\eta|^{6} |\nabla  w|^{2} + 8 s^{7} \lambda ^{8}
\xi^{7} |\nabla \eta|^{8} w^{2} ] dxdt
\\ \notag
& \quad \!\!+\! 2 \mathbb{E}\! \int_{\Sigma}
\!\Big[ - 4 s \lambda  \xi \frac{\partial
\eta}{\partial \nu} \! \Big(  \frac{\partial
\Delta w}{\partial \nu}  \Big)^{2} \!+\! 2 s
\lambda  \xi \frac{\partial \eta}{\partial \nu}
| \nabla \Delta w |^{2}\! -\! 90 s^{5}
\lambda^{5}  \xi^{5} \! \Big( \frac{\partial
\eta}{\partial \nu} \Big) ^{5} \! \Big(
\frac{\partial w}{\partial \nu}  \Big)^{2} \!+\!
12 s^{3} \lambda^{3}  \xi^{3} \! \Big (
\frac{\partial \eta}{\partial \nu} \Big) ^{3} |
\nabla^{2} w\! \cdot\! \nu |^{2}
\\ \notag
&\quad  \quad\quad\quad \ \ - 4 s^{3}
\lambda^{3}  \xi^{3}  \Big( \frac{\partial
\eta}{\partial \nu} \Big) ^{3}  \frac{\partial
w}{\partial \nu}   \frac{\partial \Delta
w}{\partial \nu} - 2 s^{3} \lambda^{3}  \xi^{3}
\Big( \frac{\partial \eta}{\partial \nu} \Big)
^{3} | \nabla^{2} w |^{2}  \Big] d\sigma dt -
(\mathcal{A} +  \mathcal{B}).
\end{align}

 {\bf Step 4}. In this step, we deal with the boundary terms.

Denote
    \begin{align}
        \notag
        \overline{B} =& \mathbb{E}\!\int_{\Sigma}\!  \Big[
            \!-\! 4 s \lambda  \xi \frac{\partial \eta}{\partial \nu} \Big(  \frac{\partial \Delta w}{\partial \nu}  \Big)^{2}\!\!
            +\! 2 s \lambda  \xi \frac{\partial \eta}{\partial \nu} | \nabla \Delta w |^{2}
            \!-\! 90 s^{5} \lambda^{5}  \xi^{5} \Big( \frac{\partial \eta}{\partial \nu} \Big) ^{5} \Big(  \frac{\partial w}{\partial \nu}  \Big)^{2}
            \!+\! 12 s^{3} \lambda^{3}  \xi^{3}  \Big( \frac{\partial \eta}{\partial \nu} \Big) ^{3} | \nabla^{2} w \nu |^{2}
            \\
            \label{eq.thm31B}
            & \quad \quad \
        \!-\! 4 s^{3} \lambda^{3}  \xi^{3}  \Big( \frac{\partial \eta}{\partial \nu} \Big) ^{3}  \frac{\partial w}{\partial \nu}   \frac{\partial \Delta w}{\partial \nu}
        \!-\! 2 s^{3} \lambda^{3}  \xi^{3}  \Big( \frac{\partial \eta}{\partial \nu} \Big) ^{3} | \nabla^{2} w |^{2}  \Big] d\sigma dt .
    \end{align}
    From \eqref{eq.estimatesOrder} and  \eqref{eq.thm31S2},     for $ s\geq CT $, we have
    \begin{align}
        \notag
        &\mathbb{E}   \int_{\Sigma} 12 s^{3} \lambda^{3}  \xi^{3}  \Big( \frac{\partial \eta}{\partial \nu} \Big) ^{3} | \nabla^{2} w   \nu |^{2} d \sigma dt
        \\ \notag
            & =  \mathbb{E} \int_{\Sigma}\! \Big[
                   48 s^{5} \lambda^{5}  \xi^{5}  \Big( \frac{\partial \eta}{\partial \nu} \Big) ^{5} \Big( \frac{\partial w}{\partial \nu} \Big)^{2}\!
                 -\!48 s^{4} \lambda^{4}  \xi^{4}  \Big( \frac{\partial \eta}{\partial \nu} \Big) ^{4} \operatorname{div}(\nu) \Big( \frac{\partial w}{\partial \nu} \Big)^{2}                 
                 \\ \label{eq.thm31BE1}
                & \quad \quad \quad \ \ \!+\! 12 s^{3} \lambda^{3}  \xi^{3}  \Big( \frac{\partial \eta}{\partial \nu} \Big) ^{3} \operatorname{div}(\nu)^2 \Big( \frac{\partial w}{\partial \nu} \Big)^{2}
                 + 12 s^{3} \lambda^{3}  \xi^{3}  \Big( \frac{\partial \eta}{\partial \nu} \Big) ^{3} \Big| \frac{\partial^{2} w}{\partial \tau \partial \nu} \Big|^{2}
                \Big] d \sigma dt
            \\ \notag
            & \geq \mathbb{E} \int_{\Sigma} \Big[
                48 s^{5} \lambda^{5}  \xi^{5}  \Big( \frac{\partial \eta}{\partial \nu} \Big) ^{5} \Big( \frac{\partial w}{\partial \nu} \Big)^{2}
                + 12 s^{3} \lambda^{3}  \xi^{3}  \Big( \frac{\partial \eta}{\partial \nu} \Big) ^{3} \Big| \frac{\partial^{2} w}{\partial \tau \partial \nu} \Big|^{2}
            \Big] d \sigma dt-\mathcal{B}_{1},
    \end{align}
    where
    \begin{equation*}
        \tau= \frac{\nabla^{2}w \nu - ( \nabla^{2} w \nu \cdot \nu  ) \nu}{|\nabla^{2}w \nu - ( \nabla^{2} w \nu \cdot \nu  ) \nu|}.
    \end{equation*}

    By  \eqref{eq.estimatesOrder} and  \eqref{eq.thm31S2}, using the same argument as \eqref{eq.thm31BE1}, for $ s\geq CT $, we have
    \begin{align}
        \notag
        &\mathbb{E}   \int_{\Sigma} - 2 s^{3} \lambda^{3}  \xi^{3}  \Big( \frac{\partial \eta}{\partial \nu} \Big) ^{3} | \nabla^{2} w |^{2} d \sigma dt\\\notag
            & =  \mathbb{E}\! \int_{\Sigma}\! \Big[
                 - 8 s^{5} \lambda^{5}  \xi^{5}  \Big( \frac{\partial \eta}{\partial \nu} \Big) ^{5}\! \Big( \frac{\partial w}{\partial \nu} \Big)^{2}
                 \!+ 8 s^{4} \lambda^{4}  \xi^{4}  \Big( \frac{\partial \eta}{\partial \nu} \Big) ^{4} \operatorname{div}(\nu)  \Big( \frac{\partial w}{\partial \nu} \Big)^{2}
                 \!- 2 s^{3} \lambda^{3}  \xi^{3}  \Big( \frac{\partial \eta}{\partial \nu} \Big) ^{3}\! \operatorname{div}(\nu)^2 \Big( \frac{\partial w}{\partial \nu} \Big)^{2}
                 \\\label{eq.thm31BE2}
                 & \quad  \quad \quad \quad \!
                 - 4 s^{3} \lambda^{3}  \xi^{3}  \Big( \frac{\partial \eta}{\partial \nu} \Big) ^{3} \Big| \frac{\partial^{2} w}{\partial \tau \partial \nu} \Big|^{2}
                 -2  \Big( \frac{\partial w}{\partial \nu} \Big)^{2} \sum_{i, j=1}^n |\tau^{i}_{x_{j}} \nu^{i}\nu^{j}|^{2}
                 \Big] d \sigma dt
            \\ \notag
            & \geq \mathbb{E} \int_{\Sigma} \Big[
                - 8 s^{5} \lambda^{5}  \xi^{5}  \Big( \frac{\partial \eta}{\partial \nu} \Big) ^{5} \Big( \frac{\partial w}{\partial \nu} \Big)^{2}
                - 4 s^{3} \lambda^{3}  \xi^{3}  \Big( \frac{\partial \eta}{\partial \nu} \Big) ^{3} \Big| \frac{\partial^{2} w}{\partial \tau \partial \nu} \Big|^{2}
            \Big] d \sigma dt-\mathcal{B}_{1}.
    \end{align}

By  \eqref{eq.estimatesOrder} , for $ s\geq CT $, direct computation gives
    \begin{align}
        \notag &
    \mathbb{E}   \int_{\Sigma} 2 s \lambda  \xi \frac{\partial \eta}{\partial \nu} | \nabla \Delta w |^{2} d \sigma dt\\
     \notag  & =  \mathbb{E} \int_{\Sigma} \Big[
                 2 s \lambda  \xi  \frac{\partial \eta}{\partial \nu} \Big( \frac{\partial \Delta w}{\partial \nu} \Big)^{2}
                 + 8 s^{3} \lambda^{3}  \xi^{3}  \Big( \frac{\partial \eta}{\partial \nu} \Big) ^{3}  \Big| \frac{\partial^{2} w}{\partial \tau \partial \nu} \Big|^{2}
                 \\ \label{eq.thm31BE3}
                 & \quad \quad \quad \quad
                 + 16 s^{3} \lambda^{3}  \xi^{3}  \Big( \frac{\partial \eta}{\partial \nu} \Big) ^{2} \frac{\partial w}{\partial \nu} \frac{\partial^{2} \eta}{\partial \tau \partial \nu}  \frac{\partial^{2} w}{\partial \tau \partial \nu}
                 + 8 s^{3} \lambda^{3}  \xi^{3}  \frac{\partial \eta}{\partial \nu}  \Big(  \frac{\partial w}{\partial \nu}  \Big)^{2}  \Big| \frac{\partial^{2} \eta}{\partial \tau \partial \nu} \Big|^{2}
                \Big] d \sigma dt
            \\ \notag 
            & \geq \mathbb{E} \int_{\Sigma} \Big[
                2 s \lambda  \xi  \frac{\partial \eta}{\partial \nu} \Big( \frac{\partial \Delta w}{\partial \nu} \Big)^{2}
                + 7 s^{3} \lambda^{3}  \xi^{3}  \Big( \frac{\partial \eta}{\partial \nu} \Big) ^{3}  \Big| \frac{\partial^{2} w}{\partial \tau \partial \nu} \Big|^{2}
            \Big] d \sigma dt-\mathcal{B}_{1}.
    \end{align}
Noting that $\frac{\partial \eta}{\partial \nu}<0$ on $\Sigma$, we have
    \begin{align}
        \label{eq.thm31BE4}
        &\mathbb{E}   \int_{\Sigma} - 4 s^{3} \lambda^{3}  \xi^{3}  \Big( \frac{\partial \eta}{\partial \nu} \Big) ^{3}  \frac{\partial w}{\partial \nu}   \frac{\partial \Delta w}{\partial \nu} d \sigma dt
        \geq \mathbb{E} \int_{\Sigma} \Big[
            16 s^{5} \lambda^{5}  \xi^{5}  \Big( \frac{\partial \eta}{\partial \nu} \Big) ^{5} \Big( \frac{\partial w}{\partial \nu} \Big)^{2}
            + s \lambda  \xi \frac{\partial \eta}{\partial \nu} \Big(  \frac{\partial \Delta w}{\partial \nu}  \Big)^{2}
        \Big] d \sigma dt.
    \end{align}
    Combining \eqref{eq.thm31B}--\eqref{eq.thm31BE4}, for $ s\geq CT $, we have
    \begin{align}
        \label{eq.thm31BE5}
        \overline{B} \geq & \mathbb{E} \int_{\Sigma}  \Big[
            -  s \lambda  \xi \frac{\partial \eta}{\partial \nu} \Big(  \frac{\partial \Delta w}{\partial \nu}  \Big)^{2}
            -34 s^{5} \lambda^{5}  \xi^{5}  \Big( \frac{\partial \eta}{\partial \nu} \Big) ^{5} \Big( \frac{\partial w}{\partial \nu} \Big)^{2}
        + 15 s^{3} \lambda^{3}  \xi^{3}  \Big( \frac{\partial \eta}{\partial \nu} \Big) ^{3}  \Big| \frac{\partial^{2} w}{\partial \tau \partial \nu} \Big|^{2}
        \Big] d\sigma dt
        - \mathcal{B}_{1}.
    \end{align}
We now estimate the last term. Thanks to the
interpolation inequality(e.g., \cite[Theorem
9.1]{Lions1972a}), \eqref{eq.eta}, \eqref{eq.alphastar}, 
Young's inequality, and noting that $\frac{\partial \eta}{\partial \nu}<0$ on $\Sigma$,  we find
\begin{align*}
\notag &\mathbb{E}  \int_{\Sigma} 15 s^{3}
\lambda^{3}  \xi^{3}  \Big( \frac{\partial
\eta}{\partial \nu} \Big) ^{3}  \Big|
\frac{\partial^{2} w}{\partial \tau \partial
\nu} \Big|^{2} d\sigma dt\\ \notag  & \geq  - C
\mathbb{E} \int_{0}^{T} s^{3} \lambda^{3}
\xi^{3} \Big| \frac{\partial w}{\partial \nu}
\Big|^{2}_{H^{1}(\partial G)}  dt
\\ 
& \geq  - C \mathbb{E} \int_{0}^{T} s^{3}
\lambda^{3} \xi^{3} \Big| \frac{\partial
w}{\partial \nu} \Big|^{6/5}_{L^{2}(\partial G)}
\Big| \frac{\partial w}{\partial \nu}
\Big|^{4/5}_{H^{5/2}(\partial G)}  dt
\\\notag
& \geq \mathbb{E} \int_{\Sigma} s^{5}
\lambda^{5}  \xi^{5}  \Big( \frac{\partial
\eta}{\partial \nu} \Big) ^{5} \Big(
\frac{\partial w}{\partial \nu} \Big)^{2} d
\sigma dt - C \mathbb{E} \int_{0}^{T} e^{2s
\alpha_{\star}} |h|^{2}_{H^{4}(G)} dt.
\end{align*}
This, together with \eqref{eq.thm31BE5}, implies that  for $ s\geq CT $,
\begin{align}
        \label{eq.thm31BE}
        \overline{B} \geq & \mathbb{E} \int_{\Sigma}  \Big[
            -  s \lambda  \xi \frac{\partial \eta}{\partial \nu} \Big(  \frac{\partial \Delta w}{\partial \nu}  \Big)^{2}
            -33 s^{5} \lambda^{5}  \xi^{5}  \Big( \frac{\partial \eta}{\partial \nu} \Big) ^{5} \Big( \frac{\partial w}{\partial \nu} \Big)^{2}
        \Big] d\sigma dt
        - C \mathbb{E} \int_{0}^{T} e^{2s \alpha_{\star}}  |h|^{2}_{H^{4}(G)} dt
        - \mathcal{B}_{1}.
    \end{align}

    By \eqref{eq.eta} and \eqref{eq.estimatesOrder},  for $ \lambda \geq C $ and  $ s\geq CT $, we have
    \begin{equation}\label{eq.thm31B1}
        -\mathcal{B}_{1} \geq \mathbb{E} \int_{\Sigma}  s^{5} \lambda^{5}  \xi^{5}  \Big( \frac{\partial \eta}{\partial \nu} \Big) ^{5} \Big( \frac{\partial w}{\partial \nu} \Big)^{2} d\sigma dt.
    \end{equation}

    Combining \eqref{eq.thm31S6}, \eqref{eq.thm31B}, \eqref{eq.thm31BE}, and  \eqref{eq.thm31B1},     for $ \lambda \geq C $ and  $ s\geq CT $,
     we get
    \begin{align}
        \notag
        & 2 \mathbb{E}\int_{Q} \theta I ( - dh + \Delta^{2} h dt ) dx
       \\ \notag
        & \geq
       2 \mathbb{E}\!\int_{Q}\! I^{2} dxdt
       + 2 \mathbb{E}\!\int_{Q}\! II_{3} dx
       -C e^{C \lambda |\eta|_{C(\overline{G})}} \mathbb{E}\! \int_{Q}\! s^{4} \lambda^{4} \xi^{4} \theta^{2} |g|^{2} dxdt
       - C \mathbb{E} \int_{0}^{T}\! e^{2s \alpha_{\star}} |h|^{2}_{H^{4}(G)} dt
       \\ \label{eq.thm31M2}
       & \quad
       + 2 \mathbb{E}\int_{Q} \big[
           3 s \lambda^{2}\xi   | \nabla \Delta w \cdot \nabla \eta  | ^{2}
           + 32 s^{3} \lambda^{4} \xi^{3}  | \nabla^{2} w  \nabla \eta \nabla \eta   |^{2}
           + 8 s^{3} \lambda^{4} \xi ^{3} |\nabla \eta |^{4} | \Delta w|^{2}
           \\ \notag
           & \qquad\qquad
           + 22 s^{5} \lambda^{6} \xi^{5} |\nabla \eta|^{4} | \nabla w \cdot \nabla \eta|^{2}
           - 16 s^{5} \lambda^{6} \xi^{5} |\nabla \eta|^{6} |\nabla  w|^{2}
           + 8 s^{7} \lambda ^{8} \xi^{7} |\nabla \eta|^{8} w^{2}\big] dxdt
       \\ \notag
       & \quad
       - \mathbb{E} \int_{\Sigma}  s^{5} \lambda^{5}  \xi^{5}  \Big( \frac{\partial \eta}{\partial \nu} \Big) ^{5} \Big( \frac{\partial w}{\partial \nu} \Big)^{2} d\sigma dt
       - (\mathcal{A} +  \mathcal{B}_{2} + \mathcal{B}_{3}  )
       - C  \mathbb{E}  \int_{(0, T) \times G_{1}}s^{5} \lambda ^{5} \xi^{5} |\nabla w|^{2} dxdt.
   \end{align}

 {\bf Step 5}. In this step,   we sort out the results of the previous steps and get the estimate of $ h $.

First, we have
    \begin{align}
        \notag
        &  \mathbb{E}\int_{Q} \big[
                8 s^{3} \lambda^{4} \xi ^{3} |\nabla \eta |^{4} | \Delta w|^{2}
                - 16 s^{5} \lambda^{6} \xi^{5} |\nabla \eta|^{6} |\nabla  w|^{2}
                + 8 s^{7} \lambda ^{8} \xi^{7} |\nabla \eta|^{8} w^{2}
            \big] dxdt
            \\ \label{eq.thm31S5E1}
            &= 2 \mathbb{E} \int_{Q} \big[8 s^{-1} \xi^{-1} (
                s^{4} \lambda^{4} \xi ^{4} |\nabla \eta |^{4} | w|
                +   s^{2} \lambda^{2} \xi ^{2} |\nabla \eta |^{2} | \Delta w|
             ) ^{2}
             - 16 s^{5} \lambda^{6} \xi ^{5} |\nabla \eta |^{6} w \Delta w
            \\ \notag
            & \quad \quad \quad \quad \
             - 16 s^{5} \lambda^{6} \xi ^{5} |\nabla \eta |^{6} |\nabla  w|^{2}
             \big] dxdt
            \\\notag
            &\geq 2 \mathbb{E} \int_{Q} \big[  8 s^{-1} \xi^{-1} (
                s^{4} \lambda^{4} \xi ^{4} |\nabla \eta |^{4} | w|
                +   s^{2} \lambda^{2} \xi ^{2} |\nabla \eta |^{2} | \Delta w|
             )^{2}\big]- \mathcal{A}_{1}.
    \end{align}
   Noting that $ ( a + b + c )^{2}  \leq 8 a^{2} + 2b^{2} + 3c^{2} $ for $a,b,c\in\dbR$, we find
    \begin{align}
        \notag
        &2 \mathbb{E}  \int_{Q}  s^{3} \lambda^{4} \xi ^{3} \theta^{2} |\nabla \eta |^{4} | \Delta h|^{2}  dxdt
        \\\notag
        &=
        2 \mathbb{E} \int_{Q}    s^{-1} \xi^{-1} \big(
                s^{2} \lambda^{2} \xi ^{2} |\nabla \eta |^{2}  \Delta w
                - 2 s^{3} \lambda^{3} \xi ^{3} |\nabla \eta |^{2} \nabla \eta \nabla w
                + s^{4} \lambda^{4} \xi ^{4} |\nabla \eta |^{4} w
                - s^{3} \lambda^{4} \xi ^{3} |\nabla \eta |^{4} w
                \\\label{eq.thm31S5E2}
                &
                \quad  \quad  \quad  \quad  \quad \quad \quad \quad
                - s^{3} \lambda^{3} \xi ^{3} |\nabla \eta |^{2} \Delta \eta w
            \big)^{2} dxdt
        \\\notag
        & \leq
        \begin{aligned}[t]
            2 \mathbb{E} \int_{Q}  [ & 8 s^{-1} \xi^{-1}  (
            s^{4} \lambda^{4} \xi ^{4} |\nabla \eta |^{4} | w|
            +   s^{2} \lambda^{2} \xi ^{2} |\nabla \eta |^{2} | \Delta w|
         ) ^{2}
         + 8 s^{5} \lambda^{6} \xi^{5} |\nabla \eta|^{4} | \nabla w \cdot \nabla \eta|^{2}
          ] dxdt
         + \mathcal{A}_{1}.
        \end{aligned}
    \end{align}

    Using \eqref{eq.pfthmFI.5} and \eqref{eq.estimatesOrder},  we obtain that  for $ s\geq C(T^{1/2}+T) $,
    \begin{align}
        \label{eq.thm31S5E3}
        2 \mathbb{E} \int_{Q} II_{3}dx  \geq  -  \mathbb{E} \int_{Q} I^{2} dx dt  - \mathcal{A}.
    \end{align}

    Thanks to \eqref{eq.eta}, we get
    \begin{equation}\label{eq.thm31S5E6}
        \mathbb{E}  \int_{(0, T) \times G_{1}} s^{5} \lambda ^{5} \xi^{5} |\nabla w|^{2} dxdt
        \leq C    \mathbb{E}  \int_{(0, T) \times G_{1}}   (s^{7} \lambda^{7} \xi^{7}\theta^{2}| h|^{2}+s^{5} \lambda^{5} \xi^{5}\theta^{2}|\nabla h|^{2}  )dx dt.
    \end{equation}

    For the left hand side of \eqref{eq.thm31M2}, we have
    \begin{equation}\label{eq.thm31S5E7}
         2 \mathbb{E}\int_{Q} \theta I ( - dh + \Delta^{2} h dt ) dx
        \leq \mathbb{E} \int_{Q} I^{2} dxdt + \mathbb{E} \int_{Q} \theta^{2} f^{2} dx dt.
    \end{equation}

    From \eqref{eq.thm31M2}--\eqref{eq.thm31S5E7},  for $ \lambda \geq C $ and $ s\geq C(T^{1/2}+T) $, we obtain that
    \begin{align}
        \notag
        &\mathbb{E} \int_{Q} \theta^{2} f^{2} dx dt  + C e^{C \lambda |\eta|_{C(\overline{G})}} \mathbb{E} \int_{Q} s^{4} \lambda^{4} \xi^{4} \theta^{2} |g|^{2} dxdt
        +  C \Big(\mathcal{A}+\mathcal{B}_{2}+\mathcal{B}_{3}
        +  |e^{s \alpha_{\star}} h|_{L_{\mathbb{F}}^{2}(0, T ; H^{4}(G))}^{2}  \Big)
            \\\label{eq.thm31S5E8}
            & +C\Big[ \mathbb{E} \int_{(0, T) \times G_{1}} \big (s^{7} \lambda^{7} \xi^{7}\theta^{2}| h|^{2}
            +s^{5} \lambda^{5} \xi^{5}\theta^{2}|\nabla h|^{2}
            +s^{3} \lambda^{4} \xi^{3}\theta^{2}|\Delta h|^{2}  \big )d x d t\Big]   \\
            \notag
            & \geq  s^{3} \lambda^{4} \mathbb{E} \int_{Q} \xi^{3} \theta^{2}|\Delta h|^{2} d x d t.
    \end{align}
 Thanks to Remark \ref{remark.LaplaceCarleman}, for $ \lambda \geq  C $ and $ s\geq C(T^{1/2}+T) $, we have
\begin{align}
    \notag
    &s^{4}  \lambda^{6} \mathbb{E} \int_{Q}  \xi^{4}  \theta^{2} (
            s^{2}  \lambda^{2} \xi^{2}|h|^{2}
            +| \nabla h|^{2}
        )dx dt
    \\\label{eq.thm31S5E8_1}
    &
    \leq
    C \Big(
            \mathbb{E} \int_{(0,T)\times G_{1}} s^{6}  \lambda^{8} \xi^{6} \theta^{2} |h|^{2} dx dt
            + \mathbb{E} \int_{Q} s^{3}  \lambda^{4} \xi^{3} \theta^{2} |\Delta h|^{2} dx dt
            \Big).
\end{align}
Hence, combining \eqref{eq.estimatesOrder}, \eqref{eq.pfthm2.1.S8}, \eqref{eq.thm31S5E8} and \eqref{eq.thm31S5E8_1},   for $ \lambda \geq  C $ and $ s\geq C(T^{1/2}+T) $, we obtain
    \begin{align}
        \notag
        & \mathbb{E} \int_{Q} \theta^{2} f^{2} dx dt  + C e^{C \lambda |\eta|_{C(\overline{G})}} \mathbb{E} \int_{Q} s^{4} \lambda^{4} \xi^{4} \theta^{2} |g|^{2} dxdt
        +C \Big(\mathcal{A}_{3}+\mathcal{B}_{2}+\mathcal{B}_{3}+|e^{s \alpha_{\star}} h|_{L_{\mathbb{F}}^{2}(0, T ; H^{4}(G))}^{2} \Big)
        \\\label{eq.thm31S5E8_2}
        & +C\Big[ \mathbb{E} \int_{(0, T) \times G_{1}}  (s^{7} \lambda^{8} \xi^{7}\theta^{2}| h|^{2}+s^{5} \lambda^{5} \xi^{5}\theta^{2}|\nabla h|^{2} +s^{3} \lambda^{4} \xi^{3}\theta^{2}|\Delta h|^{2}   )d x d t\Big]   \\
        \notag
        & \geq
         \mathbb{E} \int_{Q}  (s^{6} \lambda^{8}  \xi^{6} \theta^{2}| h|^{2}+s^{4} \lambda^{6}  \xi^{4} \theta^{2}|\nabla h|^{2} +s^{3} \lambda^{4}  \xi^{3} \theta^{2}|\Delta h|^{2}   ) d x d t.
    \end{align}

    With the help of Lemma \ref{lemma.auxiliaryEstimates} and \eqref{eq.estimatesOrder}, we can absorb $ \mathcal{B}_{2}, \mathcal{B}_{3} $, and
    $ |e^{s \alpha_{\star}} h|_{L_{\mathbb{F}}^{2}(0, T ; H^{4}(G))}^{2} $.

    Thanks to \eqref{eq.thm31S5E8_2}, for $ \lambda \geq  C $ and $ s\geq C(T^{1/2}+T) $, we have
    \begin{align}
        \notag
        & C  \mathbb{E} \int_{Q} \theta^{2} f^{2} dx dt  + C e^{C \lambda |\eta|_{C(\overline{G})}} \mathbb{E} \int_{Q} s^{4} \lambda^{4} \xi^{4} \theta^{2} |g|^{2} dxdt+C \mathcal{A}_{3}
         \\\label{eq.thm31S5E10_1}
            &
            +C\Big[ \mathbb{E} \int_{(0, T) \times G_{1}}  (s^{7} \lambda^{8} \xi^{7}\theta^{2}| h|^{2}+s^{5} \lambda^{5} \xi^{5}\theta^{2}|\nabla h|^{2} +s^{3} \lambda^{4} \xi^{3}\theta^{2}|\Delta h|^{2}   )d x d t\Big]
            \\\notag
            &\geq
             \mathbb{E} \int_{Q}  (s^{6} \lambda^{8}  \xi^{6} \theta^{2}| h|^{2}+s^{4} \lambda^{6}  \xi^{4} \theta^{2}|\nabla h|^{2} +s^{3} \lambda^{4}  \xi^{3} \theta^{2}|\Delta h|^{2}   ) d x d t.
    \end{align}

    Next, we estimate $ \mathcal{A}_{3} $.
 Letting $ \tilde{h} = s \lambda^{2} \xi e^{s \alpha} h $, we have
    \begin{equation}\label{eq.thm31S5E11}
        \mathbb{E} \int_{Q} | \Delta \tilde{h} |^{2} dxdt
        \leq C \mathbb{E} \int_{Q}  (s^{6} \lambda^{8}  \xi^{6} \theta^{2}| h|^{2}+s^{4} \lambda^{6}  \xi^{4} \theta^{2}|\nabla h|^{2} + s^{2} \lambda^{4}  \xi^{2} \theta^{2}| \Delta h|^{2} )dxdt.
    \end{equation}
    Since $ \tilde{h} =0 $ on $ \Sigma $, it follows that
    \begin{equation}\label{eq.thm31S5E12}
         |\tilde{h}|^{2}_{L^{2}_{\mathbb{F}}( 0,T; H^{2}(G) )} \leq C |\Delta \tilde{h}|^{2}_{L^{2}_{\mathbb{F}}( 0,T; L^{2}(G) )}.
    \end{equation}

    Noting that
    \begin{align*}
        \mathbb{E} \int_{Q}  s^{2} \lambda^{4}  \xi^{2} \theta^{2}| \nabla^{2} h|^{2}  dxdt
        \leq
        C |\tilde{h}|^{2}_{L^{2}_{\mathbb{F}}( 0,T; H^{2}(G) )}
        + C \mathbb{E} \int_{Q}  (
            s^{6} \lambda^{8}  \xi^{6} \theta^{2}| h|^{2}
        +   s^{4} \lambda^{6}  \xi^{4} \theta^{2}|\nabla h|^{2}  )dxdt,
    \end{align*}
    by \eqref{eq.estimatesOrder}, \eqref{eq.pfthm2.1.S8}, \eqref{eq.thm31S5E10_1}--\eqref{eq.thm31S5E12}, for $ \lambda \geq  C $ and $ s\geq C(T^{1/2}+T) $, we obtain
    \begin{align}
        \notag
        &C \mathbb{E} \int_{Q} \theta^{2} f^{2} dx dt  + C e^{C \lambda |\eta|_{C(\overline{G})}} \mathbb{E} \int_{Q} s^{4} \lambda^{4} \xi^{4} \theta^{2} |g|^{2} dxdt\\\label{eq.thm31S5E13}
            & +C\Big[ \mathbb{E} \int_{(0, T) \times G_{1}}  (s^{7} \lambda^{8} \xi^{7}\theta^{2}| h|^{2}+s^{5} \lambda^{5} \xi^{5}\theta^{2}|\nabla h|^{2} +s^{3} \lambda^{4} \xi^{3}\theta^{2}|\Delta h|^{2}   )d x d t\Big]   \\\notag
            & \geq
             \mathbb{E} \int_{Q} \Big (s^{6} \lambda^{8}  \xi^{6} \theta^{2}| h|^{2}+s^{4} \lambda^{6}  \xi^{4} \theta^{2}|\nabla h|^{2} +s^{3} \lambda^{4}  \xi^{3} \theta^{2}|\Delta h|^{2}
             + s^{2} \lambda^{4}  \xi^{2} \theta^{2}| \nabla^{2} h|^{2}  \Big ) d x d t.
    \end{align}

    We are now ready to estimate the local terms on $ \nabla h $ and $ \Delta h $. Let us introduce a cut-off function $ \psi $ such that
    \begin{equation}\label{eq.pfthm2.1.S31}
        \psi \in C_{0}^{\infty}(G_{0}), \quad \psi=1 \text { in } G_{1},  \quad 0 \leq \psi \leq 1.
    \end{equation}
    Using It\^o's formula and \eqref{eq.pfthm2.1.S1}, we get
    \begin{align*}
        d ( \psi^{4} \xi^{3} \theta^{2}h^{2}  )
         = ( 2s\psi^{4} \xi^{3} \alpha_{t}+3\psi^{4}\xi^{2}\xi_{t}  )\theta^{2}h^{2} dt+  \psi^{4} \xi^{3}  \theta^{2}g^{2}dt
              +2\psi^{4} \xi^{3}  \theta^{2} hdh.
    \end{align*}
    Hence,
    \begin{align}
        \notag
        0=&\mathbb{E} \int_{Q_{0}}  ( 2s\psi^{4} \xi^{3} \alpha_{t}+3\psi^{4}\xi^{2}\xi_{t}  )\theta^{2}h^{2} dxdt
        +\mathbb{E} \int_{Q_{0}} \psi^{4} \xi^{3} \theta^{2}g^{2}dxdt
        + \mathbb{E} \int_{Q_{0}}  2\psi^{4} \xi^{3} \theta^{2} h\Delta^{2} h dxdt
        \\ \label{eq.pfthm2.1.S33}
        &
        - \mathbb{E} \int_{Q_{0}}  2\psi^{4} \xi^{3} \theta^{2} h f dxdt.
    \end{align}
    From \eqref{eq.pfthm2.1.S31} and integration by parts, it follows that
\begin{align}
\notag &\mathbb{E} \int_{Q_{0}}
s^{3}\lambda^{4}\xi^{3} \psi^{4} \theta^{2}
|\Delta h|^{2} dx dt\\
& \leq \mathbb{E} \int_{Q_{0}}  s^{3}
\lambda^{4} \xi^{3} \psi^{4}\theta^{2}
h\Delta^{2} h dxdt + C \mathbb{E} \int_{Q_{0}}
s^{5}\lambda^{6}\xi^{5} \psi^{4} \theta^{2}
|\nabla h|^{2} dxdt
\\ \label{eq.pfthm2.1.S34}
\notag &\quad + C \mathbb{E} \int_{Q_{0}}
s^{7}\lambda^{8} \xi^{7} \theta^{2} h^{2} dxdt
        + C \mathbb{E} \int_{Q}  s^{3}\lambda^{4}\xi^{3}  \theta^{2} ( |h|^{2} + |\nabla h|^{2} ) dxdt.
    \end{align}
    For $ s\geq CT $, we also have
    \begin{equation}\label{eq.pfthm2.1.S35}
         \Big| \mathbb{E} \int_{Q_{0}}  2 s^{3}\lambda^{4}\xi^{3} \psi^{4}  \theta^{2} h f dxdt  \Big|
        \leq
        C \mathbb{E} \int_{Q_{0}} s^{6}\lambda^{8} \xi^{6} \theta^{2} h^{2} dxdt
        +C\mathbb{E} \int_{Q}  \theta^{2} f^{2} dxdt 
    \end{equation}
    and
    \begin{equation}\label{eq.pfthm2.1.S36}
         \Big| \mathbb{E} \int_{Q_{0}}  s^{3}\lambda^{4}  ( 2s\psi^{4} \xi^{3} \alpha_{t}+3\psi^{4}\xi^{2}\xi_{t}  )\theta^{2}h^{2} dxdt  \Big|
        \leq
        C   \mathbb{E} \int_{Q_{0}} s^{7}\lambda^{8} \xi^{7} \theta^{2} h^{2} dxdt.
    \end{equation}

    Using integration by parts, we obtain
    \begin{align}
        \label{eq.pfthm2.1.S38}
        \mathbb{E} \int_{Q_{0}}  s^{5}\lambda^{6}\xi^{5} \psi^{4} \theta^{2} |\nabla h|^{2} dxdt
        \leq \varepsilon \mathbb{E} \int_{Q_{0}}  s^{3}\lambda^{4}\xi^{3} \psi^{4} \theta^{2} |\Delta h|^{2} dx dt +    C\varepsilon^{-1}\mathbb{E} \int_{Q_{0}} s^{7}\lambda^{8} \xi^{7} \theta^{2} h^{2} dxdt.
    \end{align}

    From \eqref{eq.pfthm2.1.S31}--\eqref{eq.pfthm2.1.S38}, letting $ \varepsilon>0 $ small enough,  for $ \lambda \geq  C $ and $ s\geq CT $, it follows that
    \begin{align}
        \notag
        &\mathbb{E}  \int_{(0,T)\times G_{1}}   (s^{5} \lambda^{5} \xi^{5}\theta^{2}|\nabla h|^{2} +s^{3} \lambda^{4} \xi^{3}\theta^{2}|\Delta h|^{2} )dxdt
        \\\label{eq.pfthm2.1.S39}
        & \leq
        C \Big[\mathbb{E} \int_{Q}  \theta^{2} f^{2} dxdt
        + \mathbb{E} \int_{Q_{0}} s^{7}\lambda^{8} \xi^{7} \theta^{2} h^{2} dxdt
        +  \mathbb{E} \int_{Q}  s^{3}\lambda^{4}\xi^{3}  \theta^{2} ( |h|^{2} + |\nabla h|^{2} ) dxdt
        \\ \notag
        & \quad
        +  e^{C \lambda |\eta|_{C(\overline{G})}} \mathbb{E} \int_{Q} s^{4} \lambda^{4} \xi^{4} \theta^{2} |g|^{2} dxdt\Big].
    \end{align}

    Combining \eqref{eq.thm31S5E13}  and \eqref{eq.pfthm2.1.S39}, for $ \lambda \geq  C $ and $ s\geq C(T^{1/2}+T) $, we deduce
    \begin{align}
        \notag
        &
        C \Big[\mathbb{E} \int_{Q} \theta^{2} f^{2} dx dt  +  e^{C \lambda |\eta|_{C(\overline{G})}} \mathbb{E} \int_{Q} s^{4} \lambda^{4} \xi^{4} \theta^{2} |g|^{2} dxdt + \mathbb{E} \int_{Q_{0}} s^{7} \lambda^{8} \xi^{7}\theta^{2}| h|^{2} d x d t\Big] \\\notag
            & \geq
             \mathbb{E} \int_{Q} \big (s^{6} \lambda^{8}  \xi^{6} \theta^{2}| h|^{2}+s^{4} \lambda^{6}  \xi^{4} \theta^{2}|\nabla h|^{2} +s^{3} \lambda^{4}  \xi^{3} \theta^{2}|\Delta h|^{2}
             + s^{2} \lambda^{4}  \xi^{2} \theta^{2}| \nabla^{2} h|^{2}  \big ) d x d t.
    \end{align}
  By a density argument, we obtain \eqref{eq.carlemanBackward}.
\end{proof}

\section{Proof of Theorem \ref{thm.observabilityForBackwardEquation}}\label{section.ObserAndControl}

In this section, we use the Carleman estimate \eqref{eq.carlemanBackward} to prove Theorem \ref{thm.observabilityForBackwardEquation}.

\begin{proof}[Proof of Theorem \ref{thm.observabilityForBackwardEquation}]
	We first consider  the case $ T \in (0,1] $.
	Thanks to \eqref{eq.carlemanBackward}, and choosing  $ \lambda=C $ and $ s= C(1+r_{1}^{2})$, we deduce that
	\begin{align}
		\label{eq.pfThm1.2S1}
		\mathbb{E} \int_{Q} \xi^{6} \theta^{2}| z |^{2}  d x d t
		\leq
		C (1+r_{1}^{2})\Big(   \mathbb{E} \int_{Q_{0}}  \xi^{7}\theta^{2}| z |^{2} d x d t
		+  \mathbb{E} \int_{Q} \xi^{4} \theta^{2} |a_{3} z + Z |^{2} dxdt \Big)
		.
	\end{align}
	Recalling \eqref{eq.alpha}, it follows from \eqref{eq.pfThm1.2S1} that
	\begin{align*}
		\notag
		&\inf\limits_{x \in G}   (e^{2s\alpha(T/4,x)} \xi^{6}(T/2,x)    )  \mathbb{E} \int_{T/4}^{3T/4} \int_{G} z^{2} d x d t  \leq \mathbb{E} \int_{Q} e^{2s\alpha} \xi^{6} z^{2} d x d t
		\\
		& \leq C \Big[ \mathbb{E} \int_{Q_{0}} e^{2s\alpha} \xi^{7} z^{2} d x d t+ \mathbb{E} \int_{Q} e^{2s\alpha} \xi^{4}(a_{3}z+Z)^{2} d x d t \Big]
		\\
		\notag
		& \leq \sup\limits_{(t,x)\in Q} \big[e^{2s\alpha(t,x)} (\xi^{7}(t,x) +  \xi^{4}(t,x) )\big]
		\cdot C(1+r_{1}^{2}) \Big [ \mathbb{E} \int_{Q_{0}} z^{2} d x d t+ \mathbb{E} \int_{Q} (a_{3}z+Z)^{2} d x d t \Big].
	\end{align*}
	Therefore,
	\begin{align}\label{eq.pfThm1.2S3}
		\mathbb{E} \int_{T/4}^{3T/4} \int_{G} z^{2} d x d t \leq C e^{C (1+r_{1}^{2})T^{-1}} \Big [ \mathbb{E} \int_{Q_{0}} z^{2} d x d t+ \mathbb{E} \int_{Q} (a_{3}z+Z)^{2} d x d t \Big ].
	\end{align}
	
	Thanks to It\^o's formula, we find that for $ 0 \leq t_{1}\leq t_{2} \leq T $,
	\begin{equation*}
		\begin{aligned}
			\mathbb{E} \int_{G} z^{2}(t_{2}) dx
			-\mathbb{E} \int_{G} z^{2}(t_{1}) dx
			&=
			\mathbb{E} \int_{t_{1}}^{t_{2}} \int_{G} d( z^2 )dx
			\geq
			- C (1+r_{1}^{2}) \mathbb{E} \int_{t_{1}}^{t_{2}} \int_{G} z^{2} dxdt
		\end{aligned}
	\end{equation*}
	Thus, in terms of the Gronwall inequality,  for $ 0 \leq t_{1}\leq t_{2} \leq T $,
	\begin{equation}\label{eq.pfThm1_2S3_2}
		\mathbb{E} \int_{G} z^{2}(t_{1}) dx \leq C e^{C T (1+r_{1}^{2})} \mathbb{E} \int_{G} z^{2}(t_{2}) dx.
	\end{equation}
	From \eqref{eq.pfThm1.2S3} and \eqref{eq.pfThm1_2S3_2}, we deduce
	\begin{equation*}
		\mathbb{E} \int_{G} z^{2}(0) d x
		\leq C e^{C  (1+r_{1}^{2}) T^{-1}}
		\Big[ \mathbb{E} \int_{Q_{0}} z^{2} d x d t+ \mathbb{E} \int_{Q} (a_{3}z+Z)^{2} d x d t \Big ].
	\end{equation*}
	
	Next, we deal with the case $ T >1 $. By the result for $ T = 1 $ shown above, we have
	\begin{align}
		\notag
		\mathbb{E} \int_{G} z^{2}(0) d x
		&
		\leq C e^{C (1+r_{1}^{2})} \bigl( | z|_{L_{\mathbb{F}}^{2}(0, 1 ; L^{2}(G_{0}))}+|a_{3} z+Z|_{L_{\mathbb{F}}^{2}(0, 1 ; L^{2}(G))} \bigr)\\
		\notag
		&
		\leq
		C e^{C (1+r_{1}^{2}) (T^{-1}+1)}  \bigl(| z|_{L_{\mathbb{F}}^{2}(0, T ; L^{2}(G_{0}))}^2+|a_{3} z+Z|_{L_{\mathbb{F}}^{2}(0, T ; L^{2}(G))}^2\bigr).
	\end{align}
	This completes the proof of Theorem \ref{thm.observabilityForBackwardEquation}.
\end{proof}

\appendix

\section{Proof of the weighted identity}\label{App-id}

\begin{proof}[Proof of Theorem \ref{thm.fundamentalIndentity}]

    It is clear that \vspace{-3mm}
    \begin{equation}\label{eq.pfthmFI.s1}
        \theta dh= dw -  s \alpha_{t} w dt,
    \end{equation}
    and \vspace{-3mm}
    \begin{align}
        \notag
    \theta \Delta^{2} h 
        & =
        \Delta^{2} w
        -4 s\lambda\xi \nabla \eta \cdot \nabla \Delta w 
        -4 s\lambda^{2}\xi  (\nabla^{2}w\nabla\eta\nabla\eta ) 
        -4 s\lambda\xi \sum\limits_{i,j=1}^{n} \eta_{x_{i}x_{j}} w_{x_{i}x_{j}} 
        + 2 s^{2}\lambda^{2}\xi^{2} |\nabla\eta|^{2}\Delta w 
        \\ \notag
        & \quad
        -2 s\lambda^{2}\xi |\nabla \eta|^{2} \Delta w  - 2s\lambda\xi  \Delta \eta \Delta w 
        + 4 s^{2}\lambda^{2}\xi^{2} ( \nabla^{2}w\nabla\eta\nabla\eta  ) 
        -4 \nabla \Delta \ell \cdot \nabla w 
        \\ \label{eq.space}
        & \quad
        +12 s^{2}\lambda^{3}\xi^{2} |\nabla\eta|^{2} \nabla\eta \nabla w 
        +8 s^{2}\lambda^{2}\xi^{2}  ( \nabla^{2}\eta \nabla \eta \nabla w  ) 
        -4 s^{3} \lambda^{3} \xi^{3} |\nabla\eta|^{2} \nabla\eta \nabla w 
        \\ \notag
        & \quad
        +4 s^{2} \lambda^{2} \xi^{2} \Delta \eta \nabla \eta \nabla w 
        + 4  (\nabla \ell \cdot \nabla \Delta \ell ) w 
        +2 |\nabla^{2} \ell|^{2} w -\Delta^{2} \ell w 
        - 6 s^{3} \lambda ^{4}\xi^{3} |\nabla\eta|^{4} w 
        \\ \notag
        & \quad
        - 4 s ^{3}\lambda ^{3}\xi^{3}  (\nabla^{2}\eta \nabla \eta \nabla \eta ) w 
        + s ^{4}\lambda ^{4}\xi^{4} |\nabla \eta|^{4} w 
        - 2 s^{3} \lambda ^{3}\xi^{3} |\nabla\eta|^{2} \Delta \eta w 
        + |\Delta \ell|^{2} w.
    \end{align}

    From \eqref{eq.pfthmFI.3}--\eqref{eq.pfthmFI.5}, \eqref{eq.pfthmFI.s1}, and \eqref{eq.space}, it follows that
    \begin{equation*}
        2\theta I (- dh+\Delta^{2} h dt)=2I ( I_{1}+I_{2}+I_{3}  ).
    \end{equation*}

    We will compute $ I I_{2} $ under the form $\sum\limits_{i=1}^{6} \sum\limits_{j=1}^{6} I_{i j}$, where $I_{i j}$ is the product of the $i$-th term of $I$ with the $j$-th term of $I_{2}$.

    By direct computations, we get the following equalities:
    \begin{align*}
        I_{11}
         &= \sum_{i,j,k,p=1}^n  \Big(  \Phi_{1}^{p} w_{x_{k}x_{k}x_{p}} w_{x_{i}x_{i}x_{j}} - \frac{1}{2} \Phi_{1}^{j} w_{x_{k}x_{k}x_{p}} w_{x_{i}x_{i}x_{p}}  \Big)_{x_{j}} dt
         - \sum_{i,j,k,p=1}^n \Phi_{1x_{j}}^{p}  w_{x_{k}x_{k}x_{p}} w_{x_{i}x_{i}x_{j}}dt
        \\
        & \quad
        + \frac{1}{2} \sum_{i,j,k,p=1}^n \Phi_{1x_{p}}^{p} w_{x_{k}x_{k}x_{j}} w_{x_{i}x_{i}x_{j}}dt,
        \\
        I_{12}
        &=\sum_{i,j,p=1}^n \Big( \Phi_{2} w_{x_{i}x_{i}x_{j}} w_{x_{p}x_{p}} 
        - \frac{1}{2} \Phi_{2x_{j}} w_{x_{p}x_{p}} w_{x_{i}x_{i}}  \Big)_{x_{j}}dt 
        - \Phi_{2}| \nabla \Delta w|^{2} dt
        + \frac{1}{2} \Delta \Phi_{2} |\Delta w|^{2} dt,
    \\
    I_{13}
        & =
        \sum_{i,j,k,p=1}^n  \Big ( \Phi_{3}^{kp}  w_{x_{k}x_{p}} w_{x_{i}x_{i}x_{j}}
        - \Phi_{3}^{kp}  w_{x_{k}x_{p}x_{j}} w_{x_{i}x_{i}}
        + \Phi_{3x_{j}}^{kp}  w_{x_{k}x_{p}} w_{x_{i}x_{i}}
        + \Phi_{3}^{kj}  w_{x_{k}x_{p}x_{p}} w_{x_{i}x_{i}}
        \\
        & \quad \quad \quad \quad \quad
        - \frac{1}{2}  \Phi_{3x_{p}}^{jp}  w_{x_{k}x_{k}} w_{x_{i}x_{i}}
        - 2  \Phi_{3x_{i}}^{kp}  w_{x_{k}x_{p}} w_{x_{i}x_{j}}
        +2  \Phi_{3x_{p}}^{kj}  w_{x_{i}x_{k}} w_{x_{i}x_{p}}
        -  \Phi_{3x_{j}}^{kp}  w_{x_{i}x_{k}} w_{x_{i}x_{p}} \Big ) _{x_{j}} dt
        \\
        & \quad
        - \sum_{i,j,k,p=1}^n  \Phi_{3}^{kp}  w_{x_{j}x_{j}x_{k}} w_{x_{i}x_{i}x_{p}}dt
        - \sum_{i,j,k,p=1}^n  \Phi_{3x_{j}x_{j}}^{kp}  w_{x_{k}x_{p}} w_{x_{i}x_{i}} dt
        +    \frac{1}{2} \sum_{k,p=1}^{n} \Phi_{3x_{k}x_{p}}^{kp}  |\Delta w|^{2}dt
        \\
        & \quad
        + 2 \sum_{i,j,k,p=1}^n  \Phi_{3x_{i}x_{j}}^{kp}  w_{x_{k}x_{p}} w_{x_{i}x_{j}}dt
        - 2 \sum_{i,j,k,p=1}^n  \Phi_{3x_{j}x_{p}}^{kp}  w_{x_{i}x_{j}} w_{x_{i}x_{k}}dt
        + \sum_{i,j,k,p=1}^n  \Phi_{3x_{j}x_{j}}^{kp}  w_{x_{i}x_{p}} w_{x_{i}x_{k}}dt.
    \\
    I_{14}   & = \sum_{i,j,k=1}^n  \Big ( \Phi_{4}^{k} w_{x_{i}x_{i}x_{j}} w_{x_{k}}
    - \Phi_{4x_{i}}^{k} w_{x_{i}x_{j}}w_{x_{k}}
    - \Phi_{4}^{k} w_{x_{i}x_{j}} w_{x_{i}x_{k}}
    + \Phi_{4x_{i}x_{j}}^{k}w_{x_{i}}w_{x_{k}}
    - \frac{1}{2} \Phi_{4x_{i}x_{k}}^{j} w_{x_{i}}w_{x_{k}}
    \\
    & \quad \quad \quad \quad  \
    + \frac{1}{2} \Phi_{4}^{j}w_{x_{i}x_{k}}^{2}\Big ) _{x_{j}} dt
    - \sum_{i,j,k=1}^n\Phi_{4x_{i}x_{j}x_{j}}^{k} w_{x_{i}}w_{x_{k}} dt
    + \sum_{i,j,k=1}^n\frac{1}{2} \Phi_{4x_{i}x_{j}x_{k}}^{k}w_{x_{i}}w_{x_{j}}dt
    \\
    & \quad
    + 2 \sum_{i,j,k=1}^{n} \Phi_{4x_{j}}^{k} w_{x_{i}x_{j}}w_{x_{i}x_{k}} dt
    - \frac{1}{2} \sum_{i,j,k=1}^{n} \Phi_{4x_{k}}^{k} w_{x_{i}x_{j}}^{2}dt,
    \\
    I_{15}
    & = \sum_{i, j=1}^n \Big( \Phi_{5} w_{x_{i}x_{i}x_{j}} w
    - \Phi_{5x_{j}} w_{x_{i}x_{i}}w + \Phi_{5x_{i}x_{i}} w_{x_{j}}w
    - \frac{1}{2} \Phi_{5x_{i}x_{i}x_{j}} w^{2}
    -\Phi_{5}w_{x_{i}x_{i}}w_{x_{j}}
    +2\Phi_{5x_{i}}w_{x_{i}}w_{x_{j}}
    \\
    & \quad \quad \quad \ \
    - \Phi_{5x_{j}}w_{x_{i}}^{2} \Big ) _{x_{j}}dt
    + \frac{1}{2} \sum_{i, j=1}^n  \Phi_{5x_{i}x_{i}x_{j}x_{j}} w^{2} dt
    -2 \sum_{i, j=1}^n \Phi_{5x_{i}x_{j}}w_{x_{i}}w_{x_{j}}dt
    +  \Phi_{5} |\Delta w|^{2}dt;
    \\
            I_{21}
            & = \frac{1}{2} \sum_{i,j,k=1}^n  ( \Psi_{2} \Phi_{1}^{j}w_{x_{i}x_{i}}w_{x_{k}x_{k}}  ) _{x_{j}}dt - \frac{1}{2} \sum_{k=1}^{n}  ( \Psi_{2} \Phi_{1}^{k}  ) _{x_{k}} |\Delta w|^{2}dt,
            \quad
            I_{22} =\Phi_{2}\Psi_{2} |\Delta w|^{2} dt,
            \\
            I_{23}
            & = \sum_{j,k,p=1}^n\Big [ \Psi_{2} \Phi_{3}^{kp} w_{x_{j}}w_{x_{k}x_{p}}
                   -  (\Psi_{2} \Phi_{3}^{jp} )_{x_{k}} w_{x_{k}} w_{x_{p}}
                   +\frac{1}{2}  (\Psi_{2} \Phi_{3}^{kp} )_{x_{j}} w_{x_{k}} w_{x_{p}}
                   - \Psi_{2} \Phi_{3}^{kj} w_{x_{p}}w_{x_{k}x_{p}}
                   \\
                   & \quad \quad \quad \ \
                   + \frac{1}{2} (\Psi_{2} \Phi_{3}^{jp} )_{x_{p}} w_{x_{k}}^{2}\Big]_{x_{j}}dt
                   + \sum_{j,k,p=1}^n  (\Psi_{2} \Phi_{3}^{kp} )_{x_{j}x_{k}} w_{x_{j}} w_{x_{p}}dt
                   - \frac{1}{2} \sum_{j,k,p=1}^n (\Psi_{2} \Phi_{3}^{kp} )_{x_{j}x_{j}} w_{x_{k}} w_{x_{p}}dt
                   \\
                   & \quad
                   -\frac{1}{2} \sum_{j,k,p=1}^n  (\Psi_{2} \Phi_{3}^{kp} )_{x_{k}x_{p}} w_{x_{j}}^{2}dt
                   + \sum_{j,k,p=1}^{n} \Psi_{2} \Phi_{3}^{kp}  w_{x_{j}x_{k}} w_{x_{j}x_{p}}dt,
        \\
        I_{24}
            & =\sum_{j,k=1}^n  \Big( \Psi_{2} \Phi_{4}^{k} w_{x_{j}}w_{x_{k}} - \frac{1}{2} \Psi_{2} \Phi_{4}^{j} w_{x_{k}}^{2}  \Big) _{x_{j}}dt- \sum_{j,k=1}^n   ( \Psi_{2} \Phi_{4}^{k}  )_{x_{j}} w_{x_{j}}w_{x_{k}}dt
            + \frac{1}{2} \sum_{j,k=1}^n   ( \Psi_{2} \Phi_{4}^{k}  )_{x_{k}} w_{x_{j}}^{2}dt,\\
            I_{25}
            &= \sum_{j=1}^n \Big[ \Psi_{2} \Phi_{5} w_{x_{j}}w - \frac{1}{2}  (\Psi_{2} \Phi_{5} )_{x_{j}} w^{2}  \Big] _{x_{j}}dt + \frac{1}{2}\sum_{j=1}^n  (\Psi_{2} \Phi_{5} )_{x_{j}x_{j}} w^{2}dt
            - \Psi_{2} \Phi_{5} |\nabla w|^{2}dt;
    \\
        I_{31}
            & =\sum_{i,j,k,p=1}^n  \Big( \Psi_{3}^{ik}\Phi_{1}^{p} w_{x_{i}x_{k}}w_{x_{p}x_{j}}
            - \frac{1}{2} \Psi_{3}^{ij}\Phi_{1}^{p} w_{x_{i}x_{k}}w_{x_{k}x_{p}}  \Big)_{x_{j}}dt
            - \sum_{i,j,k,p=1}^n ( \Psi_{3}^{ij}\Phi_{1}^{p} )_{x_{k}} w_{x_{i}x_{j}}w_{x_{k}x_{p}} dt
            \\
            & \quad
            + \frac{1}{2}\sum_{i,j,k,p=1}^n ( \Psi_{3}^{ij}\Phi_{1}^{p} )_{x_{i}} w_{x_{j}x_{k}}w_{x_{k}x_{p}} dt,
        \\
        I_{32}
        & =\sum_{i,j,p=1}^n \Big[  \Psi_{3}^{ij}\Phi_{2} w_{x_{i}}w_{x_{p}x_{p}}
        - ( \Psi_{3}^{ip}\Phi_{2}  )_{x_{p}} w_{x_{i}}w_{x_{j}}
        + \frac{1}{2}  ( \Psi_{3}^{ij}\Phi_{2}  )_{x_{i}} w_{x_{p}}^{2}
        -\Psi_{3}^{ip}\Phi_{2} w_{x_{i}}w_{x_{j}x_{p}}
        \\
        & \quad\quad\quad\quad
        + \frac{1}{2}  ( \Psi_{3}^{ip}\Phi_{2}  )_{x_{j}} w_{x_{i}} w_{x_{p}} \Big]_{x_{j}}  dt
        +\sum_{i,j,p=1}^n  ( \Psi_{3}^{ij}\Phi_{2}  )_{x_{j}x_{p}} w_{x_{i}}w_{x_{p}} dt
        - \frac{1}{2} \sum_{i,j,p=1}^n  ( \Psi_{3}^{ij}\Phi_{2}  )_{x_{i}x_{j}} w_{x_{p}}^{2} dt
        \\
        & \quad
        - \frac{1}{2} \sum_{i,j,p=1}^n  ( \Psi_{3}^{ij}\Phi_{2}  )_{x_{p}x_{p}} w_{x_{i}}w_{x_{j}} dt
        + \sum_{i,j,p=1}^n\Psi_{3}^{ij}\Phi_{2} w_{x_{i}x_{p}}w_{x_{j}x_{p}} dt,
        \\
        I_{33}
            & =\sum_{i,j,k,p=1}^n\Big[ \Psi_{3}^{ip}\Phi_{3}^{kj} w_{x_{i}x_{p}}w_{x_{k}}
            -  (\Psi_{3}^{ij}\Phi_{3}^{kp}  )_{x_{p}} w_{x_{i}}w_{x_{k}}
            + \frac{1}{2} (\Psi_{3}^{ik}\Phi_{3}^{jp}  )_{x_{p}} w_{x_{i}}w_{x_{k}}
            - \Psi_{3}^{ij}\Phi_{3}^{kp} w_{x_{i}x_{p}}w_{x_{k}}
            \\\notag
            & \quad
            +  \frac{1}{2} (\Psi_{3}^{ij}\Phi_{3}^{kp}  )_{x_{i}} w_{x_{p}}w_{x_{k}} \Big]_{x_{j}}dt
            + \sum_{i,j,k,p=1}^n  (\Psi_{3}^{ij}\Phi_{3}^{kp}  )_{x_{j}x_{p}} w_{x_{i}}w_{x_{k}} dt
            -\frac{1}{2} \sum_{i,j,k,p=1}^n  (\Psi_{3}^{ij}\Phi_{3}^{kp}  )_{x_{k}x_{p}} w_{x_{i}}w_{x_{j}} dt
            \\
            & \quad
            - \frac{1}{2} \sum_{i,j,k,p=1}^n  (\Psi_{3}^{ij}\Phi_{3}^{kp}  )_{x_{i}x_{j}} w_{x_{p}}w_{x_{k}} dt
            + \sum_{i,j,k,p=1}^n \Psi_{3}^{ij}\Phi_{3}^{kp} w_{x_{i}x_{p}}w_{x_{k}x_{j}}dt,
            \\
            I_{34}
            & = \frac{1}{2}\sum_{i,j,p=1}^n  ( \Psi_{3}^{ij}\Phi_{4}^{p} w_{x_{i}}w_{x_{p}}  )_{x_{j}}dt
            - \frac{1}{2} \sum_{i,j,p=1}^{n}  (\Psi_{3}^{ij}\Phi_{4}^{p} )_{x_{j}}w_{x_{i}}w_{x_{p}} dt,
            \\
            I_{35}
            & = \sum_{i, j=1}^n  \Big [\Psi_{3}^{ij}\Phi_{5} w_{x_{i}}w - \frac{1}{2}  ( \Psi_{3}^{ij}\Phi_{5}  )_{x_{i}} w^{2}   \Big ]_{x_{j}}dt + \frac{1}{2} \sum_{i, j=1}^n  ( \Psi_{3}^{ij}\Phi_{5}  )_{x_{i}x_{j}} w^{2} dt
            - \sum_{i, j=1}^n \Psi_{3}^{ij}\Phi_{5} w_{x_{i}}w_{x_{j}} dt;
            \\
    % \end{align*}
    % \begin{align*}
        I_{41}
        & = \sum_{i,k,j=1}^n  \Big[ \Psi_{4}^{i} \Phi_{1}^{k} w_{x_{i}} w_{x_{k}x_{j}} - \frac{1}{2}  ( \Psi_{4}^{i} \Phi_{1}^{k}  )_{x_{j}} w_{x_{i}} w_{x_{k}}  \Big]_{x_{j}}   dt
        + \frac{1}{2}  \sum_{i,k,j=1}^n  ( \Psi_{4}^{i} \Phi_{1}^{k}  )_{x_{j}x_{j}}w_{x_{i}} w_{x_{k}} dt
        \\
        & \quad
        -  \sum_{i,k,j=1}^n \Psi_{4}^{i} \Phi_{1}^{k} w_{x_{i}x_{j}} w_{x_{k}x_{j}} dt,
        \\
        I_{42}
                & = \sum_{i, j=1}^n   \Big( \Psi_{4}^{i}\Phi_{2} w_{x_{i}} w_{x_{j}}
                - \frac{1}{2} \Psi_{4}^{j}\Phi_{2} w_{x_{i}}^{2} \Big)_{x_{j}}dt
                -\sum_{i, j=1}^n   (\Psi_{4}^{i}\Phi_{2}  )_{x_{j}} w_{x_{i}} w_{x_{j}}dt
                + \frac{1}{2}  \sum_{i, j=1}^n   (\Psi_{4}^{j}\Phi_{2}  )_{x_{j}} w_{x_{i}}^{2}dt,
        \\
        I_{43}
                & = \sum_{i,j,k=1}^n   \Big( \Psi_{4}^{i}\Phi_{3}^{jk}  w_{x_{i}} w_{x_{k}} - \frac{1}{2} \Psi_{4}^{j}\Phi_{3}^{ik}  w_{x_{i}} w_{x_{k}} \Big)_{x_{j}}dt-\sum_{i,j,k=1}^{n}   (\Psi_{4}^{i}\Phi_{3}^{jk}   )_{x_{k}} w_{x_{i}} w_{x_{j}}dt
                \\
        & \quad
                 + \frac{1}{2}  \sum_{i,j,k=1}^n  (\Psi_{4}^{i}\Phi_{3}^{jk}   )_{x_{i}} w_{x_{j}} w_{x_{k}}dt,
        \\
        I_{44}&=\sum_{i, j=1}^n \Psi_{4}^{i}\Phi_{4}^{j} w_{x_{i}}w_{x_{j}}dt,
             \quad
             I_{45}  = \frac{1}{2}\sum_{j=1}^n  (  \Psi_{4}^{j}\Phi_{5} w^{2}  )_{x_{j}} dt - \frac{1}{2}\sum_{j=1}^n  ( \Psi_{4}^{j}\Phi_{5}   )_{x_{j}} w^{2}dt;
    \\
        I_{51}
        & = \sum_{i,k,j=1}^n  ( \Psi_{5}^{i} \Phi_{1}^{k} w_{x_{i}} w_{x_{k}x_{j}}  ) _{x_{j}}dt
                   - \sum_{i,k,j=1}^n ( \Psi_{5}^{i} \Phi_{1}^{k} )_{x_{j}} w_{x_{i}} w_{x_{k}x_{j}}  dt
                   -  \sum_{i,k,j=1}^n \Psi_{5}^{i} \Phi_{1}^{k} w_{x_{i}x_{j}} w_{x_{k}x_{j}} dt,
        \\
        I_{52} & = \sum_{i, j=1}^n   \Big( \Psi_{5}^{i}\Phi_{2} w_{x_{i}} w_{x_{j}}
                   - \frac{1}{2} \Psi_{5}^{j}\Phi_{2} w_{x_{i}}^{2} \Big)_{x_{j}}dt
                   -\sum_{i, j=1}^n   (\Psi_{5}^{i}\Phi_{2}  )_{x_{j}} w_{x_{i}} w_{x_{j}}dt
                   + \frac{1}{2}  \sum_{i, j=1}^n   (\Psi_{5}^{j}\Phi_{2}  )_{x_{j}} w_{x_{i}}^{2}dt,
        \\
        I_{53}  & = \sum_{i,j,k=1}^n   \Big( \Psi_{5}^{i}\Phi_{3}^{jk}  w_{x_{i}} w_{x_{k}}
                   - \frac{1}{2} \Psi_{5}^{j}\Phi_{3}^{ik}  w_{x_{i}} w_{x_{k}} \Big)_{x_{j}}dt-\sum_{i,j,k=1}^{n}   (\Psi_{5}^{i}\Phi_{3}^{jk}   )_{x_{k}} w_{x_{i}} w_{x_{j}}dt
                   \\
        & \quad
                   + \frac{1}{2}  \sum_{i, j=1}^n   (\Psi_{5}^{i}\Phi_{3}^{jk}   )_{x_{i}} w_{x_{j}} w_{x_{k}}dt,
        \\
         I_{54}&=\sum_{i, j=1}^n \Psi_{5}^{i}\Phi_{5}^{j} w_{x_{i}}w_{x_{j}}dt,
         \quad
         I_{55} = \frac{1}{2}\sum_{j=1}^n  (  \Psi_{5}^{j}\Phi_{5} w^{2} ) _{x_{j}} dt - \frac{1}{2}\sum_{j=1}^n  ( \Psi_{5}^{j}\Phi_{5}   )_{x_{j}} w^{2}dt;
    \intertext{and}
        I_{61}
               & =\sum_{i, j=1}^n  \Big[  \Psi_{6} \Phi_{1}^{i} w_{x_{i}x_{j}} w
               -  ( \Psi_{6} \Phi_{1}^{i}  )_{x_{j}} w_{x_{i}} w
               + \frac{1}{2}  ( \Psi_{6} \Phi_{1}^{j}  )_{x_{i}x_{i}} w^{2}
               - \frac{1}{2} \Psi_{6} \Phi_{1}^{j} w_{x_{i}}^{2} \Big]_{x_{j}}dt
               \\
               & \quad
               - \frac{1}{2}\sum_{i, j=1}^n  ( \Psi_{6} \Phi_{1}^{i}  )_{x_{i}x_{j}x_{j}} w^{2}dt
               + \sum_{i, j=1}^n  ( \Psi_{6} \Phi_{1}^{i}  )_{x_{j}} w_{x_{i}} w_{x_{j}}dt
               +\frac{1}{2}\sum_{i, j=1}^n  ( \Psi_{6} \Phi_{1}^{i}  )_{x_{i}} w_{x_{j}}^{2}dt
               \\
               I_{62} & = \sum_{j=1}^n \Big [ \Psi_{6}\Phi_{2}w w_{x_{j}}
                       - \frac{1}{2}  ( \Psi_{6}\Phi_{2}  )_{x_{j}}w^{2} \Big ]_{x_{j}}dt
                       +\frac{1}{2}  \sum_{j=1}^n   ( \Psi_{6}\Phi_{2}  )_{x_{j}x_{j}} w^{2}dt
                       - \sum_{j=1}^n \Psi_{6}\Phi_{2} w_{x_{j}}^{2}dt,
                       \\
            I_{63} & = \sum_{i, j=1}^n \Big [ \Psi_{6}\Phi_{3}^{ij}w w_{x_{i}}
                    - \frac{1}{2}  ( \Psi_{6}\Phi_{3}^{ij}  )_{x_{i}}w^{2} \Big ]_{x_{j}}dt
                    +\frac{1}{2}  \sum_{i, j=1}^n   ( \Psi_{6}\Phi_{3}^{ij}  )_{x_{i}x_{j}} w^{2}dt
                    - \sum_{i, j=1}^n \Psi_{6}\Phi_{3}^{ij} w_{x_{i}}w_{x_{j}}dt,
            \\
            I_{64} & =\frac{1}{2} \sum_{j=1}^n  (\Psi_{6}\Phi_{4}^{j} w^{2}   )_{x_{j}}dt
            - \frac{1}{2}  \sum_{j=1}^n  (\Psi_{6}\Phi_{4}^{j} )_{x_{j}} w^{2} dt,
            \quad
            I_{65}  =\Psi_{6}\Phi_{5} w^{2}dt.
        \end{align*}
    Using It\^o's formula, we conclude
    \begin{align*}
        I_{16} & = 
        -\sum_{i, j=1}^n  ( w_{x_{i}x_{i}x_{j}} dw - w_{x_{i}x_{j}}dw_{x_{i}}   )_{x_{j}}
        -\frac{1}{2} d  (|\nabla^{2}w|^{2} ) + \frac{1}{2} | d\nabla^{2}w   |^{2},
            \\
            I_{26}  & =- \sum_{j=1}^n ( \Psi_{2} w_{x_{j}} dw  )_{x_{j}}
                + \sum_{i=1}^n \Psi_{2_{x_{i}}} w_{x_{i}} dw
                + \frac{1}{2} d ( \Psi_{2} |\nabla w|^{2}  )
                - \frac{1}{2} \Psi_{2t} |\nabla w|^{2} dt
                -\frac{1}{2} \Psi_{2}  | d \nabla w  |^{2},
            \\
            I_{36}  & = 
                - \sum_{i, j=1}^n (\Psi_{3}^{ij} w_{x_{i}} dw  )_{x_{j}}
                +  \sum_{i, j=1}^n \Psi_{3x_{j}}^{ij} w_{x_{i}} dw
                + 2 d  ( |  s \lambda \xi \nabla \eta \nabla w  |^{2} )
                - 4 \sum_{i, j=1}^n s^{2}\lambda^{2}\xi ( \xi \eta_{x_{j}} )_{t} \eta_{x_{i}} w_{x_{i}} w_{x_{j}} dt
                \\
                       & \quad
                - 2  s^{2}\lambda^{2}\xi^{2}  | \nabla \eta d \nabla w  |^{2},
            \\
             I_{46}&= - \sum_{i=1}^n \Psi_{4}^{i} w_{x_{i}} dw, \quad
              I_{56} = - \sum_{i=1}^n \Psi_{5}^{i} w_{x_{i}} dw,
             \quad
             I_{66}   =
                -\frac{1}{2} d ( \Psi_{6} w^{2}  )
                +\frac{1}{2} \Psi_{6 t}w^{2}dt
                +\frac{1}{2} \Psi_{6} ( dw  )^{2}.
    \end{align*}

By summing all the $ I_{ij} $,  we get
\eqref{eq.finalEquation}.
\end{proof}

\section{Proof of Lemma \ref{lemma.auxiliaryEstimates}}\label{sec.lemmas}

To prove Lemma \ref{lemma.auxiliaryEstimates}, we need the following regularity estimate for the stochastic fourth order parabolic equation.

\begin{lemma}\label{lemma.AdjustBackwardEstimates}
There exists a constant $ C>0 $,  such that for all
$ (h, g) \in  L_{\mathbb{F}}^{2}(0, T ; H^{4}(G)) \times L_{\mathbb{F}}^{2}(0, T;$ $L^{2}(G)) $, and $ f \in L_{\mathbb{F}}^{2}(0, T ; L^{2}(G)) $ satisfying
\begin{equation*}
    \left\{
    \begin{alignedat}{2}
        & d h= \Delta^2 h d t  + f d t+g d W(t) && \quad \text { in } Q, \\
        & h=\Delta h=0 && \quad \text { on } \Sigma,\\
        & h(T, \cdot)=0 && \quad \text { in } G,
    \end{alignedat}
    \right.
\end{equation*}
we have
\begin{equation*}
    |h|^{2}_{ L_{\mathbb{F}}^{2}(0, T ; H^{4}(G))} \leq C |f|^{2}_{L_{\mathbb{F}}^{2}(0, T ; L^{2}(G)) }.
\end{equation*}
\end{lemma}
\begin{proof} The proof is standard. For the sake of completeness, we provide a sketch below. Define an unbounded operator $A$ on $L^{2}(G)$ as follows:
    \begin{equation*}
        \left\{\begin{aligned}
            &\mathcal{D}(A)= \{ \varphi \in H^{4}(G)  \mid  \varphi, \Delta \varphi \in H^{1}_{0}(G) \},   \\
            &A \varphi= \Delta^{2} \varphi,    \quad \forall \varphi \in \mathcal{D}(A).
        \end{aligned}\right.
    \end{equation*}
Denote by $0<\mu_1<\mu_2\leq \cdots$ the eigenvalues of $A$ and
by $\{\phi_{i}\}_{i=1}^{\infty} \subset \mathcal{D}(A) $ the corresponding eigenfunctions such that $ |\phi_{i} |_{L^{2}(G)}=1$ ($i=1, 2, 3, \cdots$), which serves as an orthonormal basis of $ L^{2}(G) $.  Thanks to the regularity for solutions to the elliptic equations (see \cite[Chapter 2, Theorem 5.1]{Lions1972a}), we have $ \phi_{i} \in H^{8}(G) $, and $ \Delta^{3} \phi_{i}=0 $ on $ \partial G $ for $ i=1,2,3, \cdots $

For $ i=1,\dots,m $, consider the following backward stochastic differential  equations
\begin{equation}\label{eq.lemma1E1}
    \begin{cases}
        d c_{i}^{m} =  \mu_{i} c_{i}^{m} dt + ( f, \phi_{i} )_{L^{2}(G)}  dt + g_{i}^{m} dW(t),\\
        c_{i}^{m}(T)=0
    \end{cases}
\end{equation}
We have $ ( f, \phi_{i} )_{L^{2}(G)}  \in L^{2}_{\mathbb{F}}(0,T;\mathbb{R})$. According to the classical theory of backward stochastic differential equations (see \cite[Theorem 4.2]{Lue2021a}),  we know that there is a unique solution  $( c_{i}^{m} (\cdot) , g_{i}^{m} (\cdot) ) \in L_{\mathbb{F}}^{2}(\Omega ; C([0, T] ; \mathbb{R})) \times L_{\mathbb{F}}^{2}(0, T ; \mathbb{R})$ for  $i=1, \ldots, m$.

Denote
\begin{equation}\label{eq.lemma1E2}
    h^{m}=\sum_{i=1}^{m} c_{i}^{m}(t) \phi_{i}.
\end{equation}
From  \eqref{eq.lemma1E1} and \eqref{eq.lemma1E2}, we get
\begin{equation*}
    d h^{m} = \Delta^{2} h^{m} dt + f^{m} dt + g^{m} dW(t),
\end{equation*}
where $ f^{m} = \sum\limits_{i=1}^{m} ( f, \phi_{i} )_{L^{2}(G)} \phi_{i} $, $ g^{m} = \sum\limits_{i=1}^{m} g_{i}^{m}\phi_{i} $ and $ h^{m}(T)\equiv 0 $.

Using It\^o's formula, we get that
\begin{equation}\label{eq.lemma1E4}
\begin{aligned}
d( h^{m} \Delta^{2} h^{m} ) &= ( \Delta^{2}
h^{m} )^{2} dt +  f^{m} \Delta^{2} h^{m}  dt +
h^{m} \Delta^{4}h^{m} dt + h^{m} \Delta^{2}f^{m}
dt + g^{m} \Delta^{2}g^{m} dt
\\
& \quad+ g^{m} \Delta^{2} h^{m} d W(t) + h^{m}
\Delta^{2} g^{m} dW(t).
\end{aligned}
\end{equation}
Integrating the equality \eqref{eq.lemma1E4} on $ Q $, taking expectation in both sides, we obtain
\begin{align}
    \label{eq.lemma1E5}
    - \int _{G} h^{m}(0) \Delta^{2} h^{m}(0) dx
    = \mathbb{E} \int_{Q} [
        ( \Delta^{2} h^{m} )^{2} +  f^{m} \Delta^{2} h^{m} +  h^{m} \Delta^{4}h^{m} +  h^{m} \Delta^{2}f^{m}  + g^{m} \Delta^{2}g^{m}
    ] dxdt.
\end{align}

Thanks to  \eqref{eq.lemma1E2}, we get
\begin{equation}\label{eq.lemma1E4_1}
- \int _{G} h^{m}(0) \Delta^{2} h^{m}(0) dx =
-\sum_{i=1}^{m} \mu_{i} ( c_{i}^{m}(0) )^{2}
\leq 0.
\end{equation}

Utilizing integration by parts, noting $ h^{m}=\Delta h^{m}=\Delta^{2} h^{m}=\Delta^{3} h^{m}= f^{m}=\Delta f^{m}=g^{m}=\Delta g^{m}=0 $ on $ \Sigma $, we have
\begin{align}
    \label{eq.lemma1E6}
    \mathbb{E} \int_{Q} [
        h^{m} \Delta^{4}h^{m} +  h^{m} \Delta^{2}f^{m}  + g^{m} \Delta^{2}g^{m}
    ] dxdt
    =
    \mathbb{E} \int_{Q} [
        (\Delta^{2} h^{m} )^{2} +   f^{m} \Delta^{2} h^{m}   +(  \Delta g^{m} )^{2}
    ] dxdt.
\end{align}

Combining \eqref{eq.lemma1E5}--\eqref{eq.lemma1E6}, using Cauchy's inequality,  we find
\begin{equation}\label{eq.lemma1E7}
    \mathbb{E} \int_{Q} | \Delta^{2} h^{m} |^{2} dxdt \leq  C \mathbb{E} \int_{Q} | f^{m} |^{2} dxdt.
\end{equation}
Thanks to elliptic regularity theory (see \cite[Chapter 4, Theorem 4.1]{Lions1972b}), we get
\begin{equation}\label{eq.lemma1E8}
    |h^{m}|^{2}_{L^{2}_{\mathbb{F}}( 0,T; H^{4}(G) )} \leq C
        |\Delta^{2} h^{m}|^{2}_{L^{2}_{\mathbb{F}}( 0,T; L^{2}(G) )}
        .
\end{equation}
Combining \eqref{eq.lemma1E7}  and \eqref{eq.lemma1E8}, we have
\begin{equation*}
    |h^{m}|^{2}_{L^{2}_{\mathbb{F}}( 0,T; H^{4}(G) )} \leq C | f^{m} |^{2} _{L^{2}_{\mathbb{F}}( 0,T; L^{2}(G) )}.
\end{equation*}
Therefore $ \{h^{m}\}_{m=1}^{\infty} $ is
bounded in  $ L^{2}_{\mathbb{F}}( 0,T; H^{4}(G)
) $, and there exists a subsequence, which is
still denoted by $ \{h^{m}\}_{m=1}^{\infty} $,
such that $ h^{m} $ weakly convergence to $
\tilde{h} \in L^{2}_{\mathbb{F} } ( 0,T;
H^{4}(G) ) $ as $m$ tends to $\infty$. On the other hand, thanks to the well-posedness of
backward stochastic evolution equations (see
\cite[Theorem 4.11]{Lue2021a}), we know that $\lim\limits_{m\to\infty}
|h^{m}- h|_{L^{2}_{\mathbb{F}}( 0,T; L^{2}(G)) }
= 0$. Hence, $ \tilde{h}=h $ for a.e.
$(t,x,\om)\in [0,T]\times G\times\Omega$ and
\begin{equation*}
    |h|_{L^{2}_{\mathbb{F} } ( 0,T; H^{4}(G) )} \leq \liminf_{m \rightarrow \infty} |h^{m}|_{L^{2}_{\mathbb{F} } ( 0,T; H^{4}(G) )} \leq C |f|_{L^{2}_{\mathbb{F} } ( 0,T; L^{2}(G) )}.
\end{equation*}
\end{proof}

\begin{proof}[Proof of Lemma \ref{lemma.auxiliaryEstimates} ]

Using the trace inequality (e.g., \cite[Theorem
9.4]{Lions1972a}) and interpolation inequality
between Sobolev spaces (e.g., \cite[Theorem
9.1]{Lions1972a}), we have
\begin{equation}\label{eq.lemma2E1}
\begin{gathered}
|\nabla^{2}h |^{2}_{L^{2}(\partial G)}   \leq
C|h |^{2}_{H^{\frac{5}{2}}(G)}\leq  C
|h|_{H^{2}(G)} |h|_{H^{3}(G)}, \\
 |\nabla^{3}h |^{2}_{L^{2}(\partial G)}  \leq C|h
|^{2}_{H^{\frac{7}{2}}(G)}\leq C |h|_{H^{3}(G)}
|h|_{H^{4}(G)}.
\end{gathered}
\end{equation}
By \eqref{eq.lemma2E1}, \eqref{eq.alphastar}, and  interpolation results between the Hilbert
spaces $L^{2}_{\mathbb{F}}(0, T ; H^{r}(G))(r
\in[0,4])$ (see \cite[Chapter 1, Remark 14.4]{Lions1972a}), we have
\begin{equation}\label{eq.lemma2E2}
    \begin{aligned}
        \mathbb{E} \int_{\Sigma}  s^{\frac{9}{4}} \lambda^{3} \xi^{\frac{9}{4}}  \theta^{2} |\nabla^{2} h|^{2} d\sigma dt
     &\leq C \mathbb{E} \int_{0}^{T} s^{\frac{9}{4}} \lambda^{3} \xi_{\star}^{\frac{9}{4}} e^{2s \alpha_{\star}} |h|^{\frac{3}{4}} _{L^{2}(G)}  |h|^{\frac{5}{4}} _{H^{4}(G)} dt \\
        & \leq C\Big( \mathbb{E} \int_{Q} s^{6} \lambda^{8} \xi^{6} \theta^{2} h^{2} dxdt
        + \mathbb{E} \int_{0}^{T} e^{2s \alpha_{\star}} |h|^{2} _{H^{4}(G)} dt\Big),
    \end{aligned}
\end{equation}
where we use $ \alpha \geq \alpha_{\star} $ and $ \xi \geq \xi_{\star} $ in $ \overline{Q} $.

Similarly, we can obtain that
\begin{align}
    \notag
    \mathbb{E} \int_{\Sigma}  s^{\frac{3}{4}} \lambda \xi^{\frac{3}{4}}  \theta^{2} |\nabla^{3} h|^{2} d\sigma dt
    &\leq C \mathbb{E} \int_{0}^{T} s^{\frac{3}{4}} \lambda \xi_{\star}^{\frac{3}{4}} e^{2s \alpha_{\star}} |h|^{\frac{1}{4}} _{L^{2}(G)}  |h|^{\frac{7}{4}} _{H^{4}(G)} dt
    \\
    \label{eq.lemma2E3}
    &
    \leq  C \Big(\mathbb{E} \int_{Q} s^{6} \lambda^{8} \xi^{6} \theta^{2} h^{2} dxdt
    + \mathbb{E} \int_{0}^{T} e^{2s \alpha_{\star}} |h|^{2} _{H^{4}(G)} dt\Big).
\end{align}

To estimate $  \mathbb{E} \int_{0}^{T} e^{2s \alpha_{\star}} |h|^{2} _{H^{4}(G)} dt $, we set $ h_{\star}=e^{s \alpha_{\star}} h$. This function satisfies
\begin{equation*}
\left\{
\begin{alignedat}{2}
& - d h_{\star} +  \Delta^2 h_{\star} d t =
\big[e^{s\alpha_{\star}} f - ( e^{s\alpha_{\star}})_{t} h \big] d t
+ e^{s\alpha_{\star}} g d W(t) && \quad \text { in } Q, \\
& h_{\star}=\Delta h_{\star}=0 && \quad \text { on } \Sigma,\\
& h_{\star}(T, \cdot)=0 && \quad \text { in } G,
\end{alignedat}
\right.
\end{equation*}
where we use $ \lim\limits_{t \rightarrow T^{-}} e^{s \alpha_{\star}} = 0 $.

Thanks to Lemma
\ref{lemma.AdjustBackwardEstimates}  and
\eqref{eq.estimatesOrder}, for   $ s\geq C
T^{1/2} $, we have
\begin{align}
    \notag
    \mathbb{E} \int_{0}^{T} e^{2s \alpha_{\star}} |h|^{2} _{H^{4}(G)} dt
    &= |h_{\star}|^{2}_{ L_{\mathbb{F}}^{2}(0, T ; H^{4}(G))}
    \leq C |e^{s\alpha_{\star}} f - ( e^{s\alpha_{\star}} )_{t} h|^{2}_{L_{\mathbb{F}}^{2}(0, T ; L^{2}(G)) }
    \\ \label{eq.lemma2E5}
& \leq C \Big( \mathbb{E} \int_{Q} \theta^{2} |f|^{2}
dxdt + \mathbb{E}\int_{Q}  s^{6} \xi^{6}
\theta^{2} h^{2} dxdt\Big).
\end{align}
Combining \eqref{eq.lemma2E2}--\eqref{eq.lemma2E5}, the proof is complete.
\end{proof}

% \bibliographystyle{siam}
% \bibliography{../bibtex/math}

\end{document}